\documentclass[11pt]{article}

\usepackage{amssymb}
\usepackage{amsmath}
\usepackage{graphicx,color}
\usepackage{amsfonts}
\usepackage[ignoreall]{geometry}
\usepackage{setspace}
\usepackage{natbib}
\usepackage{subcaption} 

\def\({{\Bigl(}}
\def\){{\Bigr)}}
\def\[{{\Bigl[}}
\def\]{{\Bigr]}}

\def\Pr{P}

\newcommand{\td}{\theta^{\dagger}}    
\newcommand{\ld}{\lambda^{\dagger}}
\newcommand{\gd}{\gamma^{\dagger}}
\newcommand{\Qd}{Q^{\dagger}}
\newcommand{\rd}{{\rho^{\dagger}_{g_{\theta_0}}}}
\newcommand{\xd}{ {\xi^{\dagger}} }

\def\M{\mathcal M}

\newtheorem{theorem}{Theorem}

\newtheorem{corollary}[theorem]{Corollary}

\newtheorem{definition}[theorem]{Definition}
\newtheorem{example}[theorem]{Example}

\newtheorem{lemma}[theorem]{Lemma}

\newtheorem{proposition}[theorem]{Proposition}
\newtheorem{remark}[theorem]{Remark}

\newenvironment{proof}[1][Proof]{\noindent\textbf{#1.} }{\
	\rule{0.5em}{0.5em}}

\begin{document}


\title{Chernoff Index for Cox Test of Separate Parametric Families}

\author{Xiaoou Li, Jingchen Liu, and Zhiliang Ying\\
	Columbia University}
\maketitle

\baselineskip=20pt


\begin{abstract}
The asymptotic efficiency of a generalized likelihood ratio test proposed by Cox is studied under the large deviations framework  for error probabilities developed by Chernoff. In particular, two separate parametric families of  hypotheses are considered \citep{cox1961tests, cox1962further}. The significance level is set such that the maximal type I and  type II error probabilities for the generalized likelihood ratio test  decay exponentially fast with the same rate. We derive the analytic form of such a rate that is also known as the Chernoff index \citep{chernoff1952measure}, a relative efficiency measure when there is no preference between the null and the alternative hypotheses.
We further extend the analysis to approximate error probabilities when the two families are not completely separated.
Discussions are provided concerning  the implications of the present result on  model selection.

\end{abstract}




\section{Introduction}

\cite{cox1961tests, cox1962further} introduced the problem of testing two separate parametric families.
Let $X_1, \dots, X_n$ be independent and identically distributed real-valued observations from a population with density $f$ with respect to some baseline measure $\mu$. Let $\{g_\theta, \ \theta\in \Theta\}$ and $\{h_\gamma, \  \gamma\in \Gamma\}$ denote two separate parametric families of density functions with respect to the same measure $\mu$.
Consider testing $H_0$: $f\in\{g_\theta, \ \theta\in \Theta\} $ against $H_1$: $f\in\{h_\gamma, \  \gamma\in \Gamma\}$.
To avoid singularity, we  assume that all the distributions in the families $g_\theta$ and $h_\gamma$ are mutually absolutely continuous so that the likelihood ratio stays away from zero and infinity. Furthermore, we assume that the model is correctly specified, that is, $f$ belongs to either the $g$-family or the $h$-family.

Recently revisiting this problem, \cite{cox2013return} mentioned
several applications such as  the one-hit and two-hit models of
binary dose-response and testing of interactions in a balanced
$2^k$ factorial experiment.
    Furthermore, this problem has been studied in econometrics  \citep{vuong1989likelihood, white1982maximum, white1982regularity, pesaran1974general, pesaran1978testing, davidson1981several}.
    For more applications of testing separate families of hypotheses, see  \cite{berrington2007interpretation} and \cite{braganca2005separate}, and the references therein.  
Furthermore, there is a discussion of model misspecification, that is, $f$ belongs to neither the $g$-family nor the $h$-family, which is beyond the current discussion.
For semiparametric models, \cite{fine2002} proposed a similar test for non-nested hypotheses under the Cox proportional hazards model assumption.

In the discussion of \cite{cox1962further}, the test statistic $l=l_g(\hat \theta) - l_h(\hat \gamma) - E_{g_{\hat\theta} } \{l_g(\hat \theta) - l_h(\hat \gamma)\}$ is considered.
The functions $l_g(\theta)$ and $l_h(\gamma)$ are the log-likelihood functions under the $g$-family and the $h$-family and $\hat\theta$ and $\hat\gamma$ are the corresponding maximum likelihood estimators.
Rigorous distributional discussions of statistic $l$ can be found in \cite{huber1967behavior} and \cite{white1982maximum,white1982regularity}.
In this paper, we consider the generalized likelihood ratio statistic
\begin{equation}\label{glr}
LR_n=\frac {\max_{\gamma\in\Gamma} \prod_{i=1}^nh_\gamma(X_i)}{\max_{\theta\in\Theta}\prod_{i=1}^ng_\theta(X_i)} = e^{l_h(\hat \gamma)-l_g(\hat \theta)}
\end{equation}
that is slightly different from Cox's approach. We are interested in the Chernoff efficiency, whose definition is provided in Section~\ref{SecChernoff},
 of the generalized likelihood ratio test.

In the hypothesis testing literature, there are several measures of asymptotic relative efficiency for simple null hypothesis against simple alternative hypothesis.
Let $n_1$ and $n_2$ be the necessary sample sizes for each of two testing procedures to perform equivalently in the sense that they admit the same type I and type II error probabilities.
Then, the limit of ratio $n_1/n_2$ in the regime that both sample sizes tend to infinity represents the asymptotic relative efficiency between these two procedures.

Relative efficiency depends on the asymptotic manner of the two types of error probabilities with large samples.
Under different asymptotic regimes, several asymptotic efficiency measures are proposed and they are summarized in Chapter 10 of \cite{serfling1980approximation}.
Under the regime of Pitman efficiency, several asymptotically equivalent tests to Cox test exist.
Furthermore, \cite{pesaran1984asymptotic} and \cite{rukhin1993bahadur} applied Bahadur's criterion of asymptotic comparison \citep{bahadur1960stochastic, bahadur1967rates} to tests for separate families  and compared different tests for lognormal against exponential distribution and for non-nested linear regressions.
There are other efficiency measures that are frequently considered, such as Kallenberg efficiency \citep{Kal83}.

In the context of testing a simple null hypothesis against a fixed simple alternative hypothesis, \cite{chernoff1952measure} introduces a measure of asymptotic efficiency for tests based on sum of independent and identically distributed observations,  a special case of which is the likelihood ratio test.
This efficiency is introduced by showing no preference between the null hypothesis and the alternative hypothesis.
The rejection region is setup such that the two types of error probabilities decay at the same exponential rate $\rho$. The rate $\rho$ is later known as the Chernoff index. A brief summary of the Chernoff index is provided in Section \ref{SecChernoff}.

The basic strategy of \cite{chernoff1952measure} is applying large deviations techniques to the log-likelihood ratio statistic and computes/approximates the probabilities of the two types of errors.
Under the situation when either the null hypothesis or the alternative hypothesis is composite, one naturally considers the generalized likelihood ratio test.
To the authors' best knowledge, the asymptotic behavior of the generalized likelihood ratio test under the Chernoff's regime remains an open problem.
This is mostly because large deviations results are not directly applicable as the test statistic is the ratio of the supremums of two random functions.
This paper fills in this void and provides a definitive conclusion of the asymptotic efficiency of the generalized likelihood ratio test under Chernoff's asymptotic regime.
We define the Chernoff index via the asymptotic decay rate of the maximal type I and type II error probabilities that is also the minimax risk corresponding to the zero-one loss function.

We compute the generalized Chernoff index of the generalized likelihood ratio
test for two separate parametric families that keep a certain
distance away from each other. That is, the Kullback-Leibler
distance between $g_\theta$ and $h_\gamma$ are bounded away from
zero for all $\theta \in \Theta$ and $\gamma \in \Gamma$. We use
$\rho_{\theta\gamma}$ to denote the Chernoff index of the
likelihood ratio test for the simple null $H_0:~f=g_\theta$
against simple alternative $H_1:~f=h_\gamma$. Under mild moment
conditions, we show that the exponential decay rate of the maximal error probabilities is simply the minimum of the one-to-one
Chernoff index $\rho_{\theta\gamma}$ over the parameter space,
that is, $\rho = \min_{\theta, \gamma}\rho_{\theta\gamma}$.  This
result suggests that the generalized likelihood ratio test is
asymptotically the minimax strategy in the sense that with the same sample
size it achieves the optimal exponential decay rate of the maximal type I
and  type II error probabilities when they decay equally fast. The
present result can also be generalized to asymptotic analysis of Bayesian model selection among two or more
families of distributions. A key technical component is to deal
with the excursion probabilities of the likelihood functions, for
which random field and non-exponential change of measure
techniques are applied. This paper also in part corresponds to the
conjecture in \cite{cox2013return} ``formal discussion of possible
optimality properties of the test statistics would, I think,
require large deviation theory'' though we consider a slightly
different statistic.

We further extend the analysis to the cases when the two families may not be completely separate, that is, one may find two sequences of distributions in each family and the two sequences converge to each other. For this case, the Chernoff index is zero. We provide asymptotic decay rate of the type I error probability under a given distribution $g_{\theta_0}$ in $H_0$. To have the problem well-posed, the minimum Kullback-Leibler divergence between $g_{\theta_0}$ and all distributions in $H_1$ has to be bounded away from zero. 
	The result is applicable to both separated and non-separated families and thus \emph{it provides a means to approximate the error probabilities of the generalized likelihood ratio test for general parametric families.}
This result has important theoretical implications in hypothesis testing, model selection, and other areas where maximum likelihood is employed.  We provide a discussion concerning variable selection for regression models.

The rest of this paper is organized as follows. We present our main results for separate families of hypotheses in Section \ref{SecMain}. Further extension to more than two families and Bayesian model selection is discussed  in Section \ref{SecExt}. Results for possibly non-separate families are presented in Section \ref{SecNonSep}.  Numerical examples are provided in Section \ref{SecNE}. Lastly a concluding remark is give in Section \ref{SecConclude}.

\section{Main results}\label{SecMain}

\subsection{Simple null against simple alternative -- a review of Chernoff index}\label{SecChernoff}

In this section we state the main results and their implications. To start with, we provide a brief review of Chernoff index for simple null versus simple alternative; then, we proceed to the case of simple null versus composite alternative; furthermore, we present the generalized Chernoff index for the composite null versus composite alternative.

Under the context of simple null hypothesis versus simple
alternative hypothesis, we have the null hypothesis $H_0:$ $f=g$
and the alternative hypothesis $H_1:$ $f=h$. We write the
log-likelihood ratio of each observation as $l^i = \log h(X_i) - \log
g(X_i).$ Then, the likelihood ratio is $LR_n = \exp(\sum_{i=1}^n
l^i).$ We use $l$ to denote the generic random variable equal in
distribution to $l^i$. We define the moment generating function
of $l$ under distribution $g$ as $M_g(z) = E_g(e^{z l})= \int \{h(x)/g(x)\}^z g(x) \mu(dx)$, which must be finite for $z\in [0,1]$ by the H\"older inequality. Furthermore, define the
rate function
$$m_g(t) = \max_z [ z t -\log\{ M_g (z)\}].$$
The following large deviations result is established by Chernoff (1952).

\begin{proposition}\label{PropChernoff}
If $-\infty < t < E_g (l)$, then $\log \Pr _g(LR_n < e^{nt}) \sim - n\times m_g(t);$
if $E_g (l) < t< \infty$, then
$\log \Pr _g(LR_n > e^{n t}) \sim - n \times m_g(t).$
\end{proposition}
We write $a_n\sim b_n$ if $a_n/b_n\to 1$ as $n\to \infty$.
The above proposition provides an asymptotic decay rate of
the type I error probability:  for any $t> E_g(l)$
$$\Pr _g(LR_n > e^{n t})= e^{-\{1+o(1)\}n \times m_g(t)}\quad \mbox{as $n\to\infty$}.$$
Similarly, we switch the roles of $g$ and $h$ and define $M_h(z)$ and $m_h(t)$ by flipping the sign of the log-likelihood ratio $l=\log g(X) - \log h(X)$ and computing the expectations under $h$.
One further defines
$\rho(t) = \min\{m_g(t),m_h(-t)\}$
that is the slower rate among the type I and type II error probabilities. A measure of  efficiency is given by
\begin{equation}\label{Chernoff}
\rho = \max_{E_g (l)< t <E_h (l)} \rho(t)
\end{equation}
that is known as the Chernoff index between $g$ and $h$.

In the decision framework, we consider the zero-one loss function
\begin{equation}\label{loss}
L(C, f, X_1,...,X_n) = \left\{
\begin{array}{ll}
1& \quad \mbox{if $f=g$ and $(X_1,...,X_n)\in C$} \\
1& \quad \mbox{if $f=h$ and $(X_1,...,X_n)\notin C$}\\
0& \quad \mbox{otherwise}
\end{array}
\right.
\end{equation}
where $C\subset R^n$ and $f$ is a density function.
Then, the risk function is
\begin{equation}\label{risk}
R(C,f) = E_f\{L(C, f, X_1,...,X_n)\} =
\left\{\begin{array}{ll}
\Pr _g(C) & \mbox{\quad if $f=g$}\\
\Pr _h(C^c) & \mbox{\quad if $f=h$}\\
\end{array}\right. .
\end{equation}
The Chernoff index is the asymptotic exponential decay rate of the minimax risk $\min_C \max_f R(C,f)$ within the family of tests.
In the following section, we will generalize the Chernoff efficiency following the minimaxity definition.

Using the fact that $M_g(z) = M_h(1-z)$, one can show that the optimization in \eqref{Chernoff} is solved at $t=0$ and
\begin{equation}\label{rho}
\rho= \rho(0).
\end{equation}
Both $m_g(t)$ and $m_h(-t)$ are monotone functions of $t$ and \eqref{rho} suggests that $\rho=m_g(0)= m_h(0)$.
To achieve the Chernoff index, we reject the null hypothesis if the likelihood ratio statistic is greater than 1 and the type I and type II error probabilities have identical exponential decay rate $\rho$.


\begin{figure}[htb]
\begin{center}
\includegraphics[scale = .3]{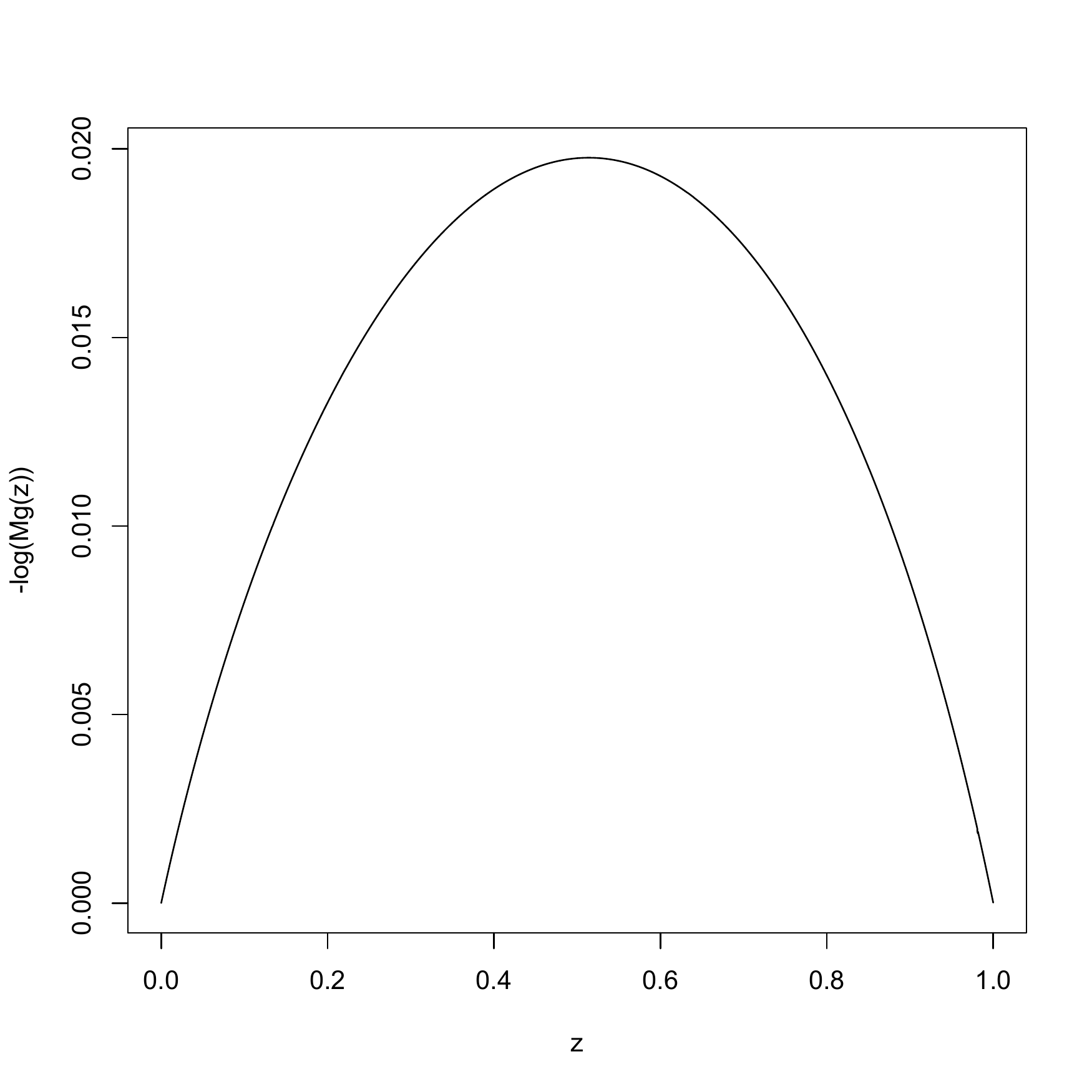}
\caption{Plot of $-\log\{M_g(z)\}$ ($y$-coordinate) against $z$ ($x$-coordinate) for the example of lognormal distribution versus exponential distribution\label{FigL}}
\end{center}
\end{figure}

To have a more concrete idea of the above calculations, Figure \ref{FigL} shows one particular  $-\log\{M_g(z)\}$ as a function of $z$ where $g(x)$ is a lognormal distribution and $h(x)$ is an exponential distribution. There are several useful facts. First, $-\log\{M_g(z)\}$ is a concave function of $z$ and $-\log\{M_g(0)\}=-\log\{M_g(1)\}=0$. The maximization $\max_z [ z t -\log\{ M_g (z)\}]$ is solved at $d\log\{ M_g (z)\}/dz = t$. Furthermore, the Chernoff index is achieved at $t=0$. We insert $t=0$ into the maximization and the Chernoff index is $\rho = \max_z [-\log\{ M_g (z)\}]$.

\subsection{Generalized Chernoff index for testing composite hypothesis}

In this subsection, we develop the corresponding results for testing composite hypotheses. Some technical conditions are required as follows.

\begin{itemize}
\item[A1] Complete separation: $\min_{\theta \in \Theta, \gamma
\in \Gamma} E_{g_\theta} \{\log g_\theta(X) - \log h_\gamma(X)\}
>0.$
\item[A2] The parameter spaces $\Theta$ and $\Gamma$ are compact
subsets of $R^{d_g}$ and $R^{d_h}$ with continuously differentiable boundary $\partial \Theta$ and $\partial \Gamma$, respectively. 
\item[A3] Define $l_{\theta\gamma}  = \log h_\gamma(X) - \log g_\theta(X)$, $S_1 =\sup_{\theta,\gamma}|\nabla_\theta l_{\theta\gamma}|$, and $S_2 = \sup_{\theta,\gamma}|\nabla_\gamma l_{\theta\gamma}|.$
There exists some $\eta,x_0>0$, that are independent with $\theta$ and $\gamma$, such that for $x> x_0$
\begin{equation}\label{cond}
\sup_{\theta\in\Theta,\gamma\in\Gamma}\max \{\Pr _{g_\theta}(S_i > x)  , \Pr _{h_\gamma}(S_i > x)\} \leq e^{-(\log x)^{1+\eta}}, \quad (i=1,2).
\end{equation}

\end{itemize}

\begin{remark}
Condition A3 requires certain tail conditions of $S_i$.
It excludes some singularity cases. 
This condition is satisfied by most parametric families. For instance, if 
$g_\theta (x) = g_0(x)e^{\theta x - \varphi_g(\theta)}$ and $h_\gamma = h_0(x)e^{\gamma x - \varphi_h(\gamma)}$ are exponential families, then 
$$|\nabla_\theta l_{\theta\gamma}| = |x - \varphi_g'(\theta)| \leq |x| + O(1).$$
Thus \eqref{cond} is satisfied if $|x|$ has a finite moment generating function.

If $g_{\theta} = g(x - \theta)$ is the scale family, 
then $$|\nabla_\theta l_{\theta\gamma}| = \Big|\frac{ g'(x - \theta)}{g(x-\theta)}\Big|$$ that usually has finite moment generating function for light-tailed distributions (Gaussian, exponential, etc) and is usually bounded for heavy-tailed distributions (e.g.~$t$-distribution). Similarly, one may verify \eqref{cond} for scale families. Thus, A3 is a weak condition and is applicable to most parametric families practically in use. 
\end{remark}

We start the discussion  for a simple null hypothesis against a composite alternative hypothesis
\begin{equation}\label{SC}
H_0:~f=g \quad \mbox {and}\quad H_1:~f\in \{h_\gamma: \gamma\in \Gamma\}.
\end{equation}
In this case, the likelihood ratio takes the following form
\begin{equation}\label{LRSC}
LR_n=\frac {\max_{\gamma\in\Gamma} \prod_{i=1}^nh_\gamma(X_i)}{\prod_{i=1}^ng(X_i)}.
\end{equation}
For each distribution $h_\gamma $ in the alternative family, we
define $\rho_\gamma$ to be the Chernoff index of the likelihood
ratio test for $H_0:~f=g$ against $H_1:~f=h_\gamma$, whose form is
given as in \eqref{Chernoff}. The first result is given as
follows.

\begin{lemma}\label{ThmSimpleComposite}
Consider the hypothesis testing problem given as in \eqref{SC} and
the generalized likelihood ratio test with rejection region
$C_\lambda = \{(x_1,...,x_n): LR_n > \lambda\}$ where $LR_n$ is given by
\eqref{LRSC}. If conditions A1-3 are satisfied and  we choose
$\lambda = 1$, then the asymptotic decay rate of the  type I and
maximal type II error probabilities are identical, more precisely,
$$\log \Pr _g( C_1)\sim \sup_{\gamma\in \Gamma}\log  \Pr _{h_\gamma} (C_1^c) \sim -n \times\min_\gamma \rho_\gamma.$$
\end{lemma}

For  composite null versus composite alternative
\begin{equation}\label{CC}
H_0:~f\in\{g_\theta: \theta\in \Theta\} \quad \mbox {against}\quad H_1:~f\in \{h_\gamma: \gamma\in \Gamma\}
\end{equation}
similar results can be obtained.
The generalized likelihood ratio statistic is given by \eqref{glr}. For each single pair $(g_\theta,h_\gamma)$, we let $\rho_{\theta \gamma}$  denote the corresponding Chernoff index of the likelihood ratio test for $H_0: f=g_\theta$ and $H_1: f= h_\gamma$. The following theorem states the main result.

\begin{theorem}\label{ThmCC}
Consider a composite null hypothesis  against a composite alternative hypothesis given as in \eqref{CC} and the generalized likelihood ratio test with rejection region $C_\lambda = \{(x_1,...,x_n): LR_n > \lambda\}$ where $LR_n$ is given by \eqref{glr}. If conditions A1-3 are satisfied and we choose $\lambda = 1$, then the asymptotic decay rate of the maximal type I and type II error probabilities are identical, more precisely,
\begin{equation}\label{riskcc}
\sup _{\theta\in \Theta} \log \Pr _{g_\theta}( C_1)\sim \sup_{\gamma\in \Gamma}\log  \Pr _{h_\gamma} (C_1^c) \sim -n \times\min_{\theta\in \Theta,\gamma\in \Gamma} \rho_{\theta\gamma}.\end{equation}
\end{theorem}

We call
$$\rho= \min_{\theta,\gamma} \rho_{\theta\gamma}$$
the generalized Chernoff index between the two families $\{g_\theta\}$ and $\{h_\gamma\}$ that is the exponential decay rate of the maximal type I and type II error probabilities for the generalized likelihood ratio test.
We would like to make a few remarks.
Suppose that $\rho_{\theta\gamma}$ is minimized at $\theta_*$ and $\gamma_*$.
The maximal type I and type II error probabilities of $C_1$ have identical exponential decay rate as that of the error probabilities of the likelihood ratio test for the simple null $H_0: f=g_{\theta_*}$ versus simple alternative $H_1: f=h_{\gamma_*}$ problem.
Then, according to the Neyman-Pearson lemma, we have the following statement.
Among all the tests for \eqref{CC} that admit maximal type I error probabilities that decays exponentially at least at rate $\rho$, their maximal type II error probabilities decay at most at rate $\rho$.
This asymptotic efficiency can only be obtained at the particular threshold $\lambda=1$, at which the maximal type I and the type II error probabilities decay exponentially equally fast.
Consider the loss function as in \eqref{loss} and the risk function is
\begin{equation}\label{risk}
R(C,f) =  \left\{\begin{array}{ll}
\Pr _f(C) & \mbox{\quad if $f\in \{g_\theta: \theta\in \Theta\}$}\\
\Pr _f(C^c) & \mbox{\quad if $f\in \{h_\gamma: \gamma \in \Gamma\}$}\\
\end{array}\right. .
\end{equation}
According to the above discussion, the maximum risk of the rejection region $C_1=\{LR_n > 1\}$ achieves the same asymptotic decay rate as that of the minimax risk that is
\begin{equation*}
\min_{C\subset R^n}\max_{f\in \{g_\theta\}\cup \{h_\gamma\}}\frac{   \log \{R(C,f)\}} n \to - \rho.
\end{equation*}

Upon considering the exponential decay rate of the two types of
error probabilities, one can simply reduce the problem to  testing
$H_0: f=g_{\theta_*}$ against $H_1: f=h_{\gamma_*}$. Each of these two
distributions can be viewed as the least favorable distribution if
its own family is chosen to be the null family.
The results in Lemma \ref{ThmSimpleComposite} and Theorem \ref{ThmCC} along with their proofs suggest that the maximal type I and  type II error probabilities are achieved at $f=g_{\theta_*}$ and $f=h_{\gamma_*}$.
In addition, under the distribution $g_{\theta_*}$
and conditional on the event $C_1$, in which $H_0$ is rejected,
the maximum likelihood estimator $\hat \gamma$ converges to
$\gamma_*$; vice versa, under the distribution $f=h_{\gamma_*}$, if
$H_0$ is not rejected, the maximum likelihood estimator $\hat
\theta$ converges to $\theta_*$.

\subsection{Relaxation of the technical conditions}\label{SecNonC}

The results of Lemma \ref{ThmSimpleComposite} and Theorem \ref{ThmCC} require three  technical conditions. Condition A1 ensures that the two families are separated and it is crucial for the exponential decay of the error probabilities. Condition A2, though important for the proof, can be relaxed for most parametric families. They can be replaced by certain localization conditions for the maximum likelihood estimator. We present one as follows.

\begin{itemize}
\item [A4]
There exist parameter-dependent compact sets $A_\theta, \tilde A_\gamma\subset \Gamma$ and $B_\gamma, \tilde B_\theta\subset \Theta$ such that for all $\theta$ and $\gamma$
\begin{eqnarray}\label{weak}
\liminf _{n\to \infty}\frac 1 n \log \Pr _{g_\theta}(\hat \theta \in \tilde B_\theta^c ~ or ~ \hat \gamma \in A_\theta^c) <-\rho, \\ \liminf _{n\to \infty}\frac 1 n \log \Pr _{h_\gamma}(\hat \theta \in B_\gamma^c~or~ \hat \gamma \in \tilde A_\gamma^c) < - \rho\notag
\end{eqnarray}
where $\hat \theta$ and $\hat\gamma$ are the maximum likelihood estimators under the two families.
Condition A3 is satisfied if the maximization in the definition of $S_i$ is taken on the set $A_\theta$ and $\tilde B_\theta$ when the tail is computed under $g_\theta$ and is taken on the set $\tilde A_\gamma$ and $B_\gamma$ when the tail is computed under $h_\gamma$.
\end{itemize}
\begin{remark}
	Assumption A4 can be verified by means of large deviations of the maximum likelihood estimator; see \cite{arcones2006large}.
	Under regularity conditions, the probability that the maximum likelihood estimator deviates from the true parameter by a constant decreases exponentially. One can choose the constant large enough so that it decays at a faster rate than $\rho$ and thus Assumption 4 is satisfied.
\end{remark}

Consider the first probability in \eqref{weak} under $g_\theta$. We typically
choose $\tilde B_\theta$ to be a reasonably large compact set
containing $\theta$ and thus $\Pr _{g_\theta}(\hat \theta \in \tilde
B^c_\theta)$ decays exponentially fast at a higher rate than
$\rho$. For the choice of $A_\theta$, we first define
$$\gamma_\theta = \arg\max_{\gamma \in \Gamma} E_{g_\theta}\{\log h_\gamma (X)\}$$
that is the limit of $\hat \gamma$ under $g_\theta$. Then, we choose $A_\theta$ be a sufficiently large compact set containing $\gamma_\theta$ so that the decay rate of $\Pr _{g_\theta}(\hat \gamma \in  A^c_\theta)$ is higher than $\rho$. Similarly, we can choose $B_\gamma$ and $\tilde A_\gamma$.
Furthermore, the maximum score function for a single observation over a compact set usually has a sufficiently light tail to satisfy condition A4, for instance, $\Pr _{g_\theta}(\sup_{\theta\in \tilde B_\theta,\gamma\in A_\theta}|\nabla_\theta l_{\theta\gamma}| > x)\leq e^{-(\log x)^{1+\eta}}$.


\begin{corollary}\label{cor:non-compact}
Consider a composite null hypothesis against composite alternative hypothesis given as in \eqref{CC}.
Suppose that conditions A1 and A4  are satisfied.
Then, the asymptotic decay rates of the maximal type I and type II error probabilities are identical, more precisely,
$$\sup _{\theta\in \Theta} \log \Pr _{g_\theta}( C_1)\sim \sup_{\gamma\in \Gamma}\log  \Pr _{h_\gamma} (C_1^c) \sim - n \times\min_{\theta,\gamma} \rho_{\theta\gamma}.$$
\end{corollary}



\section{Extensions} \label{SecExt}

\subsection{ On the asymptotic behavior of Bayes factor}\label{SecBayes}

The result in Theorem \ref{ThmCC} can be further extended to the study of Bayesian model selection.
Consider the two families in \eqref{CC} each of which is endowed with a prior distribution on its own parameter space, denoted by $\phi(\theta)$ and $\varphi(\gamma)$. We use $\M$ to denote the family membership: $\M=0$ for the $g$-family and $\M=1$ for the $h$-family. Then, the Bayes factor is
\begin{equation}\label{bf}
BF=\frac{p(X_1,...,X_n|\M=1)}{p(X_1,...,X_n|\M=0)}=\frac {\int_{\gamma\in\Gamma}\varphi(\gamma) \prod_{i=1}^nh_\gamma(X_i) d \gamma}{\int_{\theta\in\Theta}\phi(\theta)\prod_{i=1}^ng_\theta(X_i)d \theta}.
\end{equation}
With a similar derivation as that of Bayesian information criterion \citep{schwarz1978estimating}, the marginalized likelihood $p(X_1,...,X_n|\M=i)$ is  the maximized likelihood multiplied by a polynomial prefactor depending on the dimension of the parameter space. Therefore, we can approximate the Bayesian factor by the generalized likelihood ratio statistic as follows
\begin{equation}\label{bayes}
\kappa^{-1} n^{-\beta}\leq \frac {BF}{LR_n} \leq \kappa n^{\beta}
\end{equation}
for some $\kappa$ and $\beta$ sufficiently large. Therefore, $\log BF = \log LR_n + O(\log n)$. Since the expectation of $\log LR_n$ is of order $n$,  the $O(\log n)$ term does not affect the exponential rate. Therefore, we have the following result.

\begin{theorem}\label{ThmBayes}
Consider two families of distributions given as in \eqref{CC} satisfying conditions A1-3.
The prior densities $\varphi$ and $\phi$ are positive and Lipschitz continuous. We select $\M=1$ if $BF >1$ and $\M = 0$
otherwise where $BF$ is given by \eqref{bf}. Then, the asymptotic
decay rate of selecting the wrong model are identical under each
of the two families. More precisely,
\begin{eqnarray*}
&&\log \int_{\theta\in \Theta}  \Pr _{g_\theta}( BF >1) \phi(\theta)d \theta\sim \sup _{\theta\in \Theta} \log \Pr _{g_\theta}( BF>1)\\
&&~~~~~\sim \log \int_{\gamma\in \Gamma} \Pr _{h_\gamma}( BF\leq 1) \varphi(\gamma)d \gamma\sim \sup_{\gamma\in \Gamma}\log  \Pr _{h_\gamma} (BF\leq 1) \sim- n \times\min_{\theta,\gamma} \rho_{\theta\gamma}.
\end{eqnarray*}
\end{theorem}

The proof of the above theorem is an application of Theorem \ref{ThmCC} and \eqref{bayes} and thus we omit it. The above result does not rely on the validity of the prior distributions. Therefore, model selection  based on Bayes factor is asymptotically efficient even if the prior distribution is misspecified. That is, the Bayes factor is calculated based on the probability measures with density functions  $\varphi$ and $\phi$ that are different from the true prior probability measures under which $\theta$ and $\gamma$ are generated.

\subsection{Extensions to more than two families}


Suppose that there are $K$ non-overlapping families $\{g_{k, \theta_k}: \theta_k \in \Theta_k\}$ for $k=1,...,K$, among which we would like to select the true family to which the distribution $f$ belongs. Let
$$L_k(\theta_k) = \prod_{i=1}^n g_{k,\theta_k} (X_i) $$
be the likelihood of family $k$.
A natural decision is to select the family that has the highest likelihood, that is,
\begin{equation*}
\hat k = \arg\max_{k=1,...,K}\sup_{\theta_k} L_k(\theta_k).
\end{equation*}
According to the results in Theorem \ref{ThmCC}, we obtain that
\begin{equation*}
\sup_{k,\theta_{k}}\log \Pr _{g_{k,\theta_{k}}}( \hat k \neq k) \sim {-n \rho}
\end{equation*}
where $\rho$ is the smallest generalized Chernoff indices, defined as in Theorem \ref{ThmCC}, among all the $(K-1)K/2$ pairs of families. To obtain the above limit, one simply considers each family $k$ as the null hypothesis and the union of the rest $K-1$ altogether as the alternative hypothesis.

With the same argument as in Section \ref{SecBayes}, we consider Bayesian model selection among the $K$ families each of which is endowed with a prior $\phi_k(\theta_k)$. Consider the marginalized maximum likelihood estimator
$$\hat k_B = \arg\max_{k} \int L_k(\theta_k)\phi_k (\theta_k) d \theta_k$$
that admits the same misclassification rate
\begin{equation*}
\sup_{k,\theta_{k}}\log \Pr _{g_{k,\theta_{k}}}( \hat k_B \neq k) \sim\sup_k \log \int \Pr _{g_{k,\theta_{k}}}( \hat k_B \neq k)\phi_k(\theta_k)d \theta_k\sim{-n \rho}.
\end{equation*}

\section{Results for possibly non-separated families}\label{SecNonSep}
\subsection{The asymptotic approximation of error probabilities}
In this section we extend the results to the cases when the $g$-family and the $h$-family are not necessarily  separated, that is,
\begin{equation}\label{nonsep}
\min_{\theta \in \Theta, \gamma \in \Gamma} E_{g_\theta} \{\log g_\theta(X) - \log h_\gamma(X)\} =0.
\end{equation}
In the case of \eqref{nonsep}, the Chernoff index is trivially zero.
We instead derive the asymptotic decay rate of the following error probabilities. For some $\theta_0 \in \Theta$ such that $$\min_\gamma E_{g_{\theta_0}} \{\log g_{\theta_0}(X) - \log h_\gamma(X)\} >0,$$ we consider the type I error probability
\begin{equation}\label{eq:D-power}
P_{g_{\theta_0}}(LR_{n}>e^{nb}) \quad \mbox{as $n\to \infty$}
\end{equation}
where $LR_{n}$ is the generalized likelihood ratio statistic as in \eqref{glr}. For $b$, we require that 
\begin{equation}\label{b}
\sup_{\gamma\in\Gamma}E_{g_{\theta_0}}\{\log h_{\gamma}(X)-\log g_{\theta_0}(X)\}<b
\end{equation}
ensuring that $P_{g_{\theta_0}}(LR_{n}>e^{nb})$ eventually converges to zero.

The statement of the theorem requires the following construction.
For each $\theta$ and $\gamma$, we first define the moment generating function of $\log h_{\gamma}(X)-\log g_{\theta}(X)-b$
\begin{equation}
M_{g_{\theta_0}}(\theta,\gamma,\lambda)= E_{g_{\theta_0}}\Big[\exp\{\lambda(\log h_{\gamma}(X)-\log g_{\theta}(X)-b)\}\Big]
\end{equation}
and  consider the optimization problem 
\begin{equation}\label{eq:D-opt}
M^{\dagger}_{g_{\theta_0}}\triangleq\inf_{\theta\in\Theta}\sup_{\gamma\in\Gamma}\inf_{\lambda\in R} M_{g_{\theta_0}}(\theta,\gamma,\lambda).
\end{equation}
Under Assumption A2, there exists at least one solution to the above optimization we assume one of the solutions is 
$$
(\theta^{\dagger},\gamma^{\dagger},\lambda^{\dagger})=\arg\inf_{\theta\in\Theta}\sup_{\gamma\in\Gamma}\inf_{\lambda\in R} M_{g_{\theta_0}}(\theta,\gamma,\lambda).
$$
Furthermore, we define a measure $Q^{\dagger}$ that is absolutely continuous with respect to $P_{g_{\theta_0}}$  
\begin{equation}\label{eq:D-Qd}
\frac{dQ^{\dagger}}{dP_{g_{\theta_0}}}=\exp\Big\{\lambda^{\dagger}(\log h_{\gamma^{\dagger}}(X)-\log g_{\theta^{\dagger}}(X)-b)\Big\}/M^{\dagger}_{g_{\theta_0}}.
\end{equation}
\begin{definition}[Solid tangent cone]
For a set $A\subset R^d$ and $x\in A$, the solid tangent cone $T_xA$ is defined as
$$
T_xA=\{y\in R^d:\exists~ y_m \mbox{ and } \lambda_m \mbox{ such that } y_m\to y \mbox{, }\lambda_m\to 0 \mbox{ as } m\to\infty \mbox{, and } x+\lambda_m y_m\in A  \}.
$$
\end{definition}

If $A$ has continuously differentiable boundary and $x\in\partial A$, then $T_xA$ consists of all the vectors in $R^{d}$ that have negative inner products with the normal vector to $\partial A$ at x pointing outside of $A$; if $x$ is in the interior of $A$, then $T_xA= R^d.$
We consider the following technical conditions for the main theorem in this section.
\begin{enumerate}
\item[A5]The moment generating function $M_{g_{\theta_0}}$ is twice differentiable at $(\td,\gd,\ld)$.
\item[A6] Under $Q^{\dagger}$, 
the the solution to the Euler condition is unique, that is, the equation with respect to $\theta$ and $\gamma$
\begin{eqnarray}\label{eq:expected-score equation-under-Qd}
E^{\Qd} \{y^\top \nabla_{\theta}\log g_{\theta}(X) \}\leq 0 \mbox{ for all } y\in T_\theta\Theta   \\E^{Q^{\dagger}}\{y^\top \nabla_{\gamma} h_{\gamma}(X)\} \leq 0 \mbox{ for all } y\in T_{\gamma}\Gamma \notag
\end{eqnarray}
has a unique solution $(\bar{\theta},\bar{\gamma})$. In addition,
$$
E^{\Qd}\{\sup_{\theta\in\Theta} |\nabla^2_{\theta} \log g_{\theta}(X) |\}<\infty\mbox{ and }E^{\Qd}\{\sup_{\gamma\in\Gamma} |\nabla^2_{\gamma} \log h_{\gamma}(X) | \}<\infty.
$$
We also assume that under measure $Q^\dag$ as $n\to\infty$,
$$
{\sqrt{n}}(\hat{\theta}-\bar \theta)=O_{\Qd}(1)  
\mbox{ and }
{\sqrt{n}}(\hat{\gamma}-\bar \gamma) = O_{\Qd}(1),
$$  
where $\hat{\theta}$ and $\hat{\gamma}$ are the maximum likelihood estimators
$$
\hat{\theta}=\arg\sup_{\theta} \sum_{i=1}^n\log g_{\theta}(X_i)\mbox{ and } \hat{\gamma}=\arg\sup_{\gamma}\sum_{i=1}^n \log h_{\gamma}(X_i), 
$$
and a random sequence $a_n=O_{\Qd}(1)$ means it is tight under measure $Q^\dag$.


\item[A7]
 We assume that $g_{\theta_0}$ does not belong to the closure of the family of distributions $\{h_{\gamma}:\gamma\in\Gamma\}$, that is,
$\inf_{\gamma\in\Gamma} D(g_{\theta_0}\|h_{\gamma})>0$.
\end{enumerate}

Assumption A6  requires $n^{-1/2}$ convergence of $\hat{\theta}$ and $\hat{\gamma}$ under $\Qd$. It also requires  the local maximum of the function $E^{Q\dag}\log g_{\theta}(X)$ and $E^{Q\dag}\log h_{\gamma}(X)$ to be unique. We elaborate  the Euler condition for $\theta\in int(\Theta)$ and $\theta\in\partial \Theta$ separately. If $\theta\in int(\Theta)$, then $T_\theta\Theta= R^{d_g}$. 
The Euler condition is equivalent to  $E^{\Qd} \nabla_{\theta}\log g_{\theta}(X)=0$, which is the usual first order condition for a local maximum. If $\theta\in\partial \Theta$, then the Euler condition requires that the directional derivative of $E^{\Qd}\{\log g_{\theta}(X)\}$ along a vector pointing towards inside $\Theta$ is non-positive. 
Assumption A7 guarantees that the probability $\lim_{n\to\infty}P_{g_{\theta_0}}(LR_{n}>e^{nb})=0$ for some $b$.

\begin{theorem}\label{thm:error-under-theta0}
 Under Assumptions A2-A3 and A5-A7, for each $b$ satisfying \eqref{b}, we have
$$
\log P_{g_{\theta_0}}(LR_{n}>e^{nb})\sim -n\times \rd,
$$
where $\rd=-\log M^{\dagger}_{g_{\theta_0}}$ and $M^{\dagger}_{g_{\theta_0}}$ is defined in \eqref{eq:D-opt}.
\end{theorem}
This theorem provides a means to approximate the type I and type II error probabilities for general parametric families.
The above results are applicable to the both cases that the two families are separated or not separated.
According to standard large deviations calculation for random walk, we have that for each $\theta\in\Theta$ and $\gamma\in\Gamma$,
$$
\log P_{g_{\theta_0}}\Big(\sum_{i=1}^n \log h_{\gamma}(X_i)-\log g_{\theta}(X_i)-nb>0\Big)\sim \inf_{\lambda}\log M_{g_{\theta_0}}(\theta,\gamma,\lambda).
$$
Theorem \ref{thm:error-under-theta0} together with the above display implies that 
\begin{eqnarray*}
\log P_{g_{\theta_0}}(LR_n>1)&\sim& \inf_{\theta}\sup_{\gamma} \log P_{g_{\theta_0}}\Big(\sum_{i=1}^n \log h_{\gamma}(X_i)-\log g_{\theta}(X_i)>nb\Big)\\
&\sim& \log P_{g_{\theta_0}}\Big(\sum_{i=1}^n \log h_{\gd}(X_i)-\log g_{\td}(X_i)>nb\Big)
\end{eqnarray*}
The exponential decay rate of the error probabilities under $g_{\theta_0}$ is the same as the exponential decay rate of the probability that $h_{\gd}$ is preferred to $g_{\td}$.

One application of Theorem \ref{thm:error-under-theta0} is to compute the power function asymptotically. Consider the fixed type I error $\alpha$ and the critical region of the generalized likelihood ratio test is determined by the quantile of a $\chi^2$ distribution, that is $\{LR_n> e^{\lambda_\alpha}\}$ where $2\lambda_\alpha$ is the $(1-\alpha)$th quantile of the $\chi^2$ distribution. This correspond to choosing $b=o(1)$. For a given alternative distribution $h_\gamma$, one can compute the type II error probability asymptotically by means of Theorem \ref{thm:error-under-theta0} switching the role of the null and the alternative families. Thus, the power function can be computed asymptotically.

\subsection{Application to model selection in generlized linear models}

We  discuss the application of Theorem \ref{thm:error-under-theta0} on model selection for generalized linear models \citep{mccullagh1989generalized}. 
Let $Y_i$ be the response of the $i$th observation and $X^{(i)}=(X_{i1},...,X_{ip})^T$ and $Z^{(i)}=(Z_{i1},...,Z_{iq})^T$ be two sets of predictors, $i=1,...,n$.
Consider a generalized linear model with canonical link function and the true conditional distribution of $Y_i$ is 
\begin{equation}\label{true}
g_i(y_i,\beta^{0})=\exp\Big\{(\beta^{0})^T X^{(i)} y_i -b((\beta^{0})^T X^{(i)}) +c(y_i)\Big\}, \qquad i=1,2,...,n,
\end{equation}
where $f(y) = e^{c(y)}$ is the base-line density, $b(\cdot)$ is the logarithm of the moment generating function,  $\beta^{0}= (\beta_{1}^0,...,\beta_{p}^0)^T$ is the vector of true regression coefficients, and $X$ is the set of true predictors. 
Let the null hypothesis be
\begin{equation}\label{null}
H_0: g_i(y_i,\beta)=\exp\Big\{\beta^{T} X^{(i)} y_i -b(\beta^{T} X^{(i)}) +c(y_i)\Big\}, \qquad i=1,2,...,n;
\end{equation}
the alternative hypothesis is
\begin{equation}\label{alter}
H_1:  h_i(y_i,\gamma)=\exp\Big\{\gamma^{T} Z^{(i)} y_i -b(\gamma^T Z^{(i)}) +c(y_i)\Big\}, \qquad i=1,2,...,n.
\end{equation}
We further assume that $H_1$ does not contain \eqref{true}.
Conditional on the covariates $X$ and $Z$, we consider the asymptotic decay rate of the type I error probability
$$P_{\beta^{0}}(LR_n\geq1),$$
where 
$
LR_n=\frac{\sup_{\gamma}\prod_{i=1}^n h_i(Y_i,\gamma)}{\sup_{\beta}\prod_{i=1}^n g_i(Y_i,\beta)}
$ is the generalized likelihood ratio.

We present the construction of the rate function as follows.
For each $\beta\in R^p$, $\gamma\in R^q$ and $\lambda \in R$, define 
\begin{equation}\label{eq:def-widetilde-rho}
\widetilde{\rho}_n(\beta,\gamma,\lambda) = \frac{1}{n}\sum_{i=1}^n \Big\{\lambda\Big[
b(\gamma^T Z^{(i)})-b(\beta^{T}X^{(i)})\Big]+b((\beta^0)^{T}X^{(i)})- b\Big( (\beta^0)^{T}X^{(i)} +\lambda(\gamma^{T}Z^{(i)}- \beta^{T}X^{(i)}) \Big )
\Big\}.
\end{equation}
Taking derivative with respect to $\lambda$, we have
\begin{equation}\label{eq:lambda-derivative}
\frac{\partial}{\partial \lambda}\widetilde{\rho}_n(\beta,\gamma,\lambda)=
\frac{1}{n}\sum_{i=1}^n \Big\{b(\gamma^T Z^{(i)}) - b(\beta^T X^{(i)}) -   b'\Big( (\beta^0)^{T}X^{(i)} +\lambda(\gamma^{T}Z^{(i)}- \beta^{T}X^{(i)}) \Big )(\gamma^{T}Z^{(i)}- \beta^{T}X^{(i)})\Big\}.
\end{equation}
According to fact that $b(\cdot)$ is a convex function, we have
$$
\limsup_{\lambda\to+\infty} \frac{\partial}{\partial \lambda}\widetilde{\rho}_n(\beta,\gamma,\lambda)<0,
$$
if $\beta^T X^{(i)}\neq \gamma^T Z^{(i)}$ for some $i$.
Define the set  $B_n\subset R^p$ such that
$$
B_n=\{\beta: \inf_{\gamma}\frac{\partial}{\partial \lambda}\widetilde{\rho}_n(\beta,\gamma,0)\geq 0 \}.
$$
Then for each $\beta\in B_n$ and $\gamma\in R^{q}$, there is a $\lambda\geq 0$ such that 
$\frac{\partial}{\partial \lambda}\widetilde{\rho}_n(\beta,\gamma,0)= 0 $. 
Thanks to the convexity of $b$, $\beta^0\in B_n$ and thus $B_n$ is never empty.
The second derivative is
\begin{equation*}
\frac{\partial^2}{(\partial \lambda)^2}\widetilde{\rho}_n(\beta,\gamma,\lambda)=
-  \frac{1}{n} \sum_{i=1}^n  b''\Big( (\beta^0)^{T}X^{(i)} +\lambda(\gamma^{T}Z^{(i)}- \beta^{T}X^{(i)}) \Big )(\gamma^{T}Z^{(i)}- \beta^{T}X^{(i)})^2<0,
\end{equation*}
if $\beta^T X^{(i)}\neq \gamma^{T} Z^{(i)}$ for some $i$.
Therefore, there is a unique solution to the maximization $\sup_{\lambda}\widetilde{\rho}_n(\beta,\gamma,\lambda)$. We further consider the optimization 
\begin{equation}\label{eq:optim-rho}
\widetilde{\rho}_n^{\dagger}=\sup_{\beta\in B_n}\inf_{\gamma}\sup_{\lambda}\widetilde \rho_n(\beta,\gamma,\lambda).
\end{equation}
We consider the following technical conditions.
\begin{itemize}
	\item [A8] For each $n$, the solution to \eqref{eq:optim-rho} exists, denoted by $(\beta_n^{\dagger},\gamma_n^{\dagger},\lambda_n^{\dagger})$. 	
	There exists a constant $\kappa_1$ such that
	$$
	\|\beta_n^{\dagger}\|\leq \kappa_1,~ \|\gamma_n^{\dagger}\|\leq \kappa_1 \mbox{ and }\lambda_n^{\dagger}\leq \kappa_1 \mbox{ for all }n.
	$$Here, $\|\cdot\|$ is the Euclidean norm.
	
	\item [A9] There exists a constant $\delta_1>0$ such that $\inf_{\gamma}\sup_{\lambda}\widetilde{\rho}_n(\beta^{0},\gamma,\lambda)>\delta_1 \mbox{ for all } n.
	$
	\item[A10] There exists a constant $\kappa_2$ such that $\|X^{(i)}\|\leq \kappa_2$ and $\|Z^{(i)}\|\leq \kappa_2$ for all $i$. Additionally, there exits $\delta_2>0$ such that for all $n$  the smallest eigenvalue of $\frac{1}{n}\sum_{i=1}^n X^{(i)} X^{(i)T}$ is bounded below by $\delta_2$.
	\item[A11] For any compact set $K\subset R$, $\inf_{u\in K}b''(u)>0$. In addition, $b(\cdot)$ is four-time continuously differentiable.
\end{itemize}
Assumption A8 requires that the solution of the optimization \eqref{eq:optim-rho} does not tend to infinity as $n$ increases, which is a mild condition. In particular, if the Kullback-Leibler divergence $D(g_i(\cdot,\beta^0)| g_i(\cdot, \beta))$ tend to infinity uniformly for all $i$ as $\|\beta\|$ goes to infinity, then $B_n$ is a bounded subset of $R^p$ and $\|\beta_n^{\dagger} \|$ is also bounded. 
Similar checkable sufficient conditions can be obtained for $\gamma^\dagger_n$ and $\lambda^\dagger_n$.


%
\begin{theorem}\label{thm:error-glm}
	Under Assumptions A8-A11, conditional on  the covariates $X^{(i)}$ and $Z^{(i)}$, $i=1,...,n$, we have
	$$
	\log P_{{\beta^{0}}}(LR_{n}\geq 1)\sim -n\times \widetilde{\rho}_n^{\dagger},
	$$
	where $\widetilde{\rho}_n^{\dagger}$ is defined in \eqref{eq:optim-rho}.
\end{theorem}

For generalized linear models, the moment generating function of likelihood ratio is
	$$
	E_{\beta^{0}}\Big\{\lambda\sum_{i=1}^n [\log h_i(Y_i,\gamma)-\log g_i(Y_i,\beta)] \Big\} = e^{-n\widetilde{\rho}_n(\beta,\gamma,\lambda)}.
	$$
Therefore, $\widetilde{\rho}_n^{\dagger}$ is a natural generalization of 
	$\rho^{\dagger}_{g_{\theta_0}}$  for the nonidentical distribution case.

Theorem \ref{thm:error-glm} provides the asymptotic rate of selecting the wrong model by maximizing the likelihood. 
The asymptotic rate  as a function of the true regression coefficients $\beta^0$ quantifies the strength of the signals. The larger the rate is, the easier it is to  select the correct variables. 
The rate  also depends on covariates. If $Z$ is highly correlated with $X$, then the rate is small.
Overall, the rate serves as an efficiency measure of selecting the true model from families that mis-specifies the model.

%
%
%
%

\section{Numerical examples}\label{SecNE}

In this section, we present  numerical examples to illustrate the asymptotic behavior of the maximal type I and type II error probabilities and the sample size tends to infinity. The first one is an example of continuous distributions and the second one is an example of discrete distributions. The third one is an example of linear regression models where the null hypotheses and alternative are not separated. In these examples, we compute the error probabilities  using importance sampling corresponding to the change of measure in the proof with sufficiently large number of Monte Carlo replications to ensure that our estimates are sufficiently accurate.

\begin{example}\label{ex:lnorm-exp}
Consider the {lognormal distribution and exponential distribution.}
For $x>0$, let \begin{eqnarray*}
g_\theta(x)=\frac{1}{x(2\pi\theta)^{1/2}}e^{-\frac{(\log x)^2}{2\theta}} \quad  \Theta=(0,+\infty),\qquad
 h_\gamma(x)=\frac{1}{\gamma}e^{-\frac{x}{\gamma}} \quad  \Gamma=(0,+\infty)
\end{eqnarray*}
be the density functions of the lognormal distribution and the exponential distribution.

For each  $\theta$ and $\gamma$, we  compute $\rho_{\theta \gamma}$ numerically. Figure \ref{ex11} shows the contour plot of $\rho_{\theta,\gamma}$.  The minimum of $\rho_{\theta \gamma}$ is $0.020$ and is obtained at $(\theta^*,\gamma^*)=(1.28,1.72)$. From the theoretical analysis, the maximal type I and type II error probabilities for the   test decay at rate $e^{-n\rho_{\theta^* \gamma^*}}$.

Figure~\ref{ex12} is the  plot of the  maximal type I and type II error probabilities as a function of the sample size for the composite versus composite test
$$H_0: f\in\{g_\theta; \theta\in \Theta\}\quad  \mbox{against}\quad H_1: f\in\{h_\gamma; \gamma\in \Gamma\}$$
and simple versus simple test
$$H_0: f=g_{\theta_*}\quad  \mbox{against}\quad H_1: f= h_{\gamma_*}.$$
We also fit a straight line to the logarithm of error probabilities against the sample sizes using least squares and the slope  is  $-0.022$. This confirms the theoretical findings. The error probabilities shown in Figure~\ref{ex12} range from $7\times 10^{-5}$ to $0.12$ and the range for sample size is from $50$ to $370$.

\end{example}

\begin{figure}[htb]
\begin{center}
\includegraphics[height=3in,width=3in]{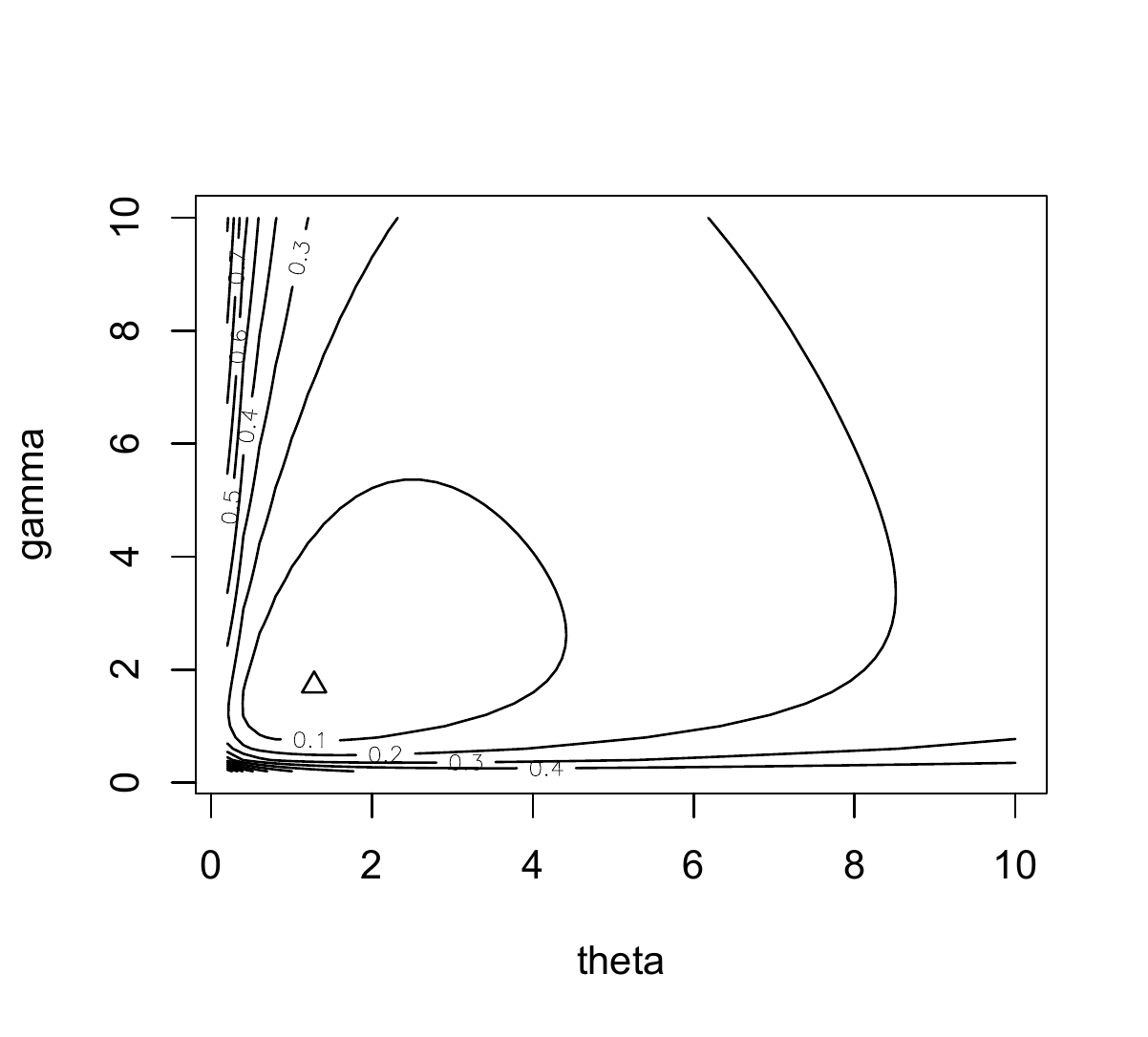}
\caption{\label{ex11}
Contour plot for $\rho_{\gamma,\theta}$ in Example~\ref{ex:lnorm-exp}. The triangle point indicates the minimum.}
\end{center}
\end{figure}

\begin{figure}[htb]
\begin{center}
\includegraphics[height=3in,width=3in]{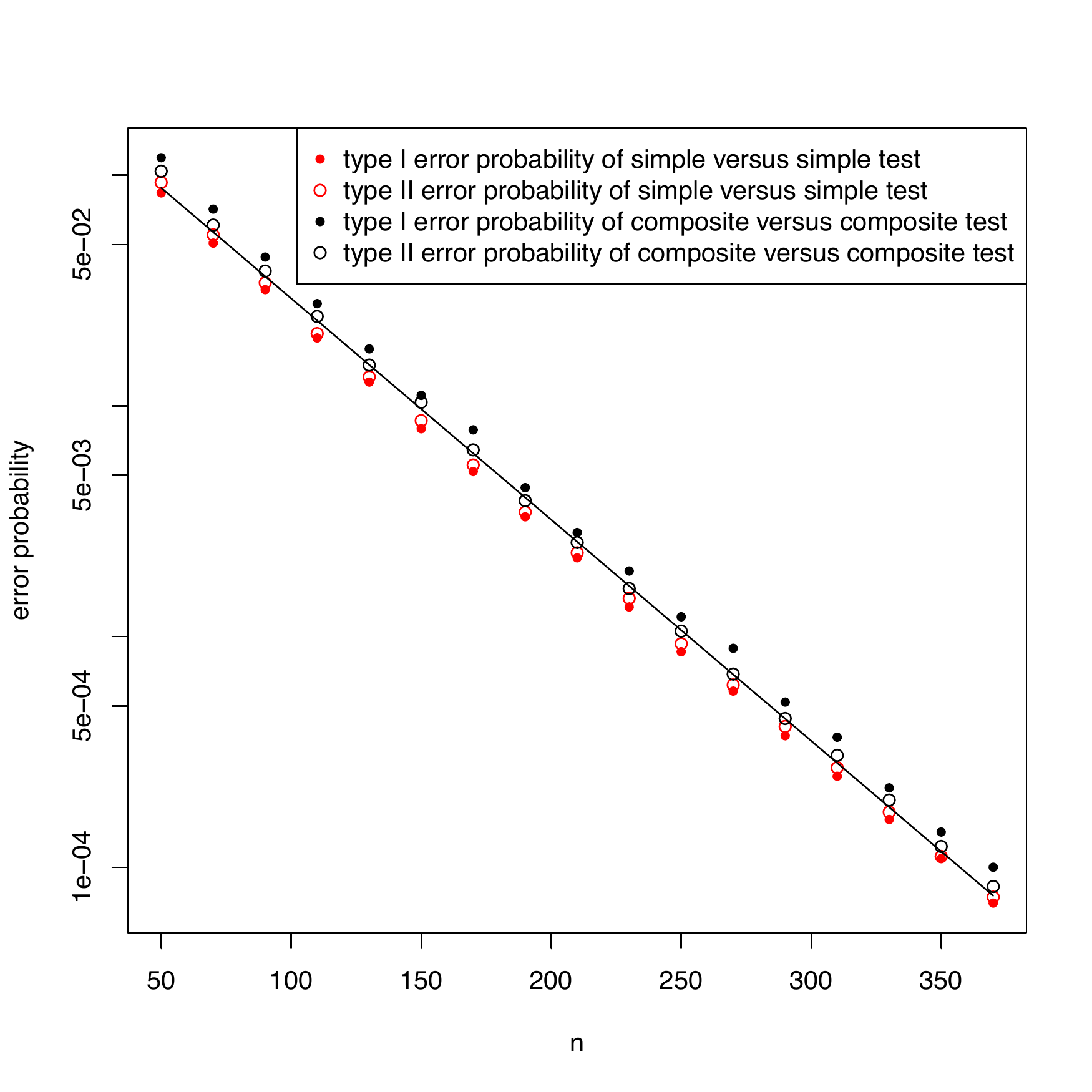}
\caption{
\label{ex12}
Decay rate of type I and type II error probabilities ($y$-coordinate) as a function of sample size ($x$-coordinate) in Example~\ref{ex:lnorm-exp}.}

\end{center}
\end{figure}

\begin{example}\label{ex:pois-geom}

We now proceed to the Poisson distribution versus the geometric distribution.
Let
\begin{eqnarray*}
g_\theta(x)=\frac{e^{-\theta}\theta^x}{x!} \quad \Theta= [1,+\infty), \qquad h_\gamma(x)=\frac{\gamma^x}{(1+\gamma)^{x+1}} \quad \Gamma=[0.5,+\infty),
\end{eqnarray*}
for $x\in \mathbb Z^+$. The parameter $\gamma$ is the failure to success odds. The minimum Chernoff index without constraint is attained at $\theta= \gamma=0$ and $\rho_{00}=0$. Thus we truncate the parameter spaces away from zero to separate the two families.

The Chernoff index $\rho_{\theta,\gamma}$ can be computed numerically and is minimized at $(\theta^*,\gamma^*)=(1,0.93)$, with $\rho_{\theta^*,\gamma^*}=0.023$.
Figure \ref{ex21} shows the contour plot of $\rho_{\theta,\gamma}$.
Same as in the previous example, we compute the maximal type I and type II error probabilities of the composite versus composite test and simple versus simple test. Figure \ref{ex22} shows the maximal  type I and type II error probabilities as a function of the sample size. The error probabilities appeared in Figure~\ref{ex22} range from $ 1.0\times 10^{-4}$ to $0.10$ with the sample sizes range from $40$ to $400$. 
We also fit a straight line to the logarithm of error probabilities against the sample sizes and the slope is $-0.025$. This numerical analysis  confirms our theorems.
\end{example}
\begin{example}\label{ex:linear}
We consider two regression models,
$$
H_0: Y=\beta_1 X_1 +\beta_2 X_2 +\varepsilon_1 \mbox{ against } H_1: Y= \beta_1 X_1+\zeta_1 Z_1+\varepsilon_2,
$$
where $(X_1,X_2,Z_1)$ jointly follows the multivariate Gaussian distribution with mean $(0,0,0)^T$ and the covariance matrix $\Sigma$. The random noises $\varepsilon_1$ and $\varepsilon_2$ follow the normal distributions $N(0,\sigma_1^2)$ and $N(0,\sigma_2^2)$ respectively and are independent of $(X_1,X_2,Z_1)$.
We assume the true model to be
$$
Y=\beta_1^0 X_1+\beta_2^0 X_2+\varepsilon,
$$
with the following parameters
$$
\beta_1^0= 1, \beta_2^0=2, \varepsilon\sim N(0,1), \mbox{ and }
\Sigma=
\begin{bmatrix}
1 & 0.1 & 0.1\\
0.1 & 1 & 0.1\\
0.1 & 0.1 &1
\end{bmatrix}.
$$
Let $(X_{i1},X_{i2},Z_{i1},Y_{i})^T$ be i.i.d.~copies of $(X_1,X_2,Z_1,Y)$ generated under the true model, for $i=1,...,n$.
Let $\theta=(\beta_1,\beta_2)$ and $\gamma=(\beta_1,\zeta_1)$ be the regression coefficients for the null and the alternative hypotheses respectively. The maximum likelihood estimators for $\theta$ and $\gamma$ are the least square estimators
$$
\hat\theta= (\tilde{X}^\top \tilde{X})^{-1}\tilde{X}^\top\tilde{Y} \mbox{ and } \hat\gamma= (\tilde{Z}^\top \tilde{Z})^{-1}\tilde{Z}^\top\tilde{Y},
$$
where 
$$
\tilde{X}=\begin{bmatrix}
X_{11} & X_{12}\\
X_{21}& X_{22}\\
\cdots & \cdots\\
X_{n1}& X_{n2}
\end{bmatrix},
\tilde{Z} =\begin{bmatrix}
X_{11} & Z_{11}\\
X_{21} & Z_{21}\\
\cdots &\cdots\\
X_{n1} & Z_{n1}
\end{bmatrix},
\mbox{ and } \tilde Y=\begin{bmatrix}
Y_{1}\\
Y_2\\
\cdots\\
Y_n
\end{bmatrix}
$$
are the design matrices for  linear models under $H_0$ and $H_1$.
We consider the error probability that the maximized log-likelihood of $H_0$ is smaller than that of $H_1$, equivalently, the residual sum of squares under $H_0$ is larger than that under $H_1$
$$
P_{\beta^0, \Sigma}\Big(\|\tilde Y - \tilde X \hat\theta\|^2> \|\tilde Y - \tilde Z \hat\gamma\|^2 \Big).
$$
From the theoretical analysis, the above probability decays at rate $e^{-n\rd}$ as $n\to\infty$, where the definition of $\rd$ is given in Theorem~\ref{thm:error-under-theta0}. 
 We solve the optimization problem \eqref{eq:D-opt} numerically and obtain $\rd=0.45$. Figure~\ref{fig:linear_error} and Figure~\ref{fig:linear_small} are  scatter plots of the error probability in the above display as a function of the sample size with different ranges for error probabilities. In Figure~\ref{fig:linear_error}, 
 the range of the error probability  is from $ 10^{-4}$ to $0.25$ and the range of sample size is from $3$ to $18$. In Figure~\ref{fig:linear_small}, the range of error probabilities is from $1.2\times 10^{-8}$ to $4.0\times 10^{-6}$ with the sample size from $24$ to $36$.
 We fit straight lines for  $\log P_{\beta^0, \Sigma}\Big(\|\tilde Y - \tilde X \hat\theta\|^2> \|\tilde Y - \tilde Z \hat\gamma\|^2 \Big)$ against $n$ using least square. The fitted slope in Figure~\ref{fig:linear_error} is $-0.52$ and the fitted slope in Figure~\ref{fig:linear_small} is $-0.47$. This confirms our theoretical results.
 
\end{example}

\begin{figure}[ht]
\begin{center}
\includegraphics[height=3in,width=3in]{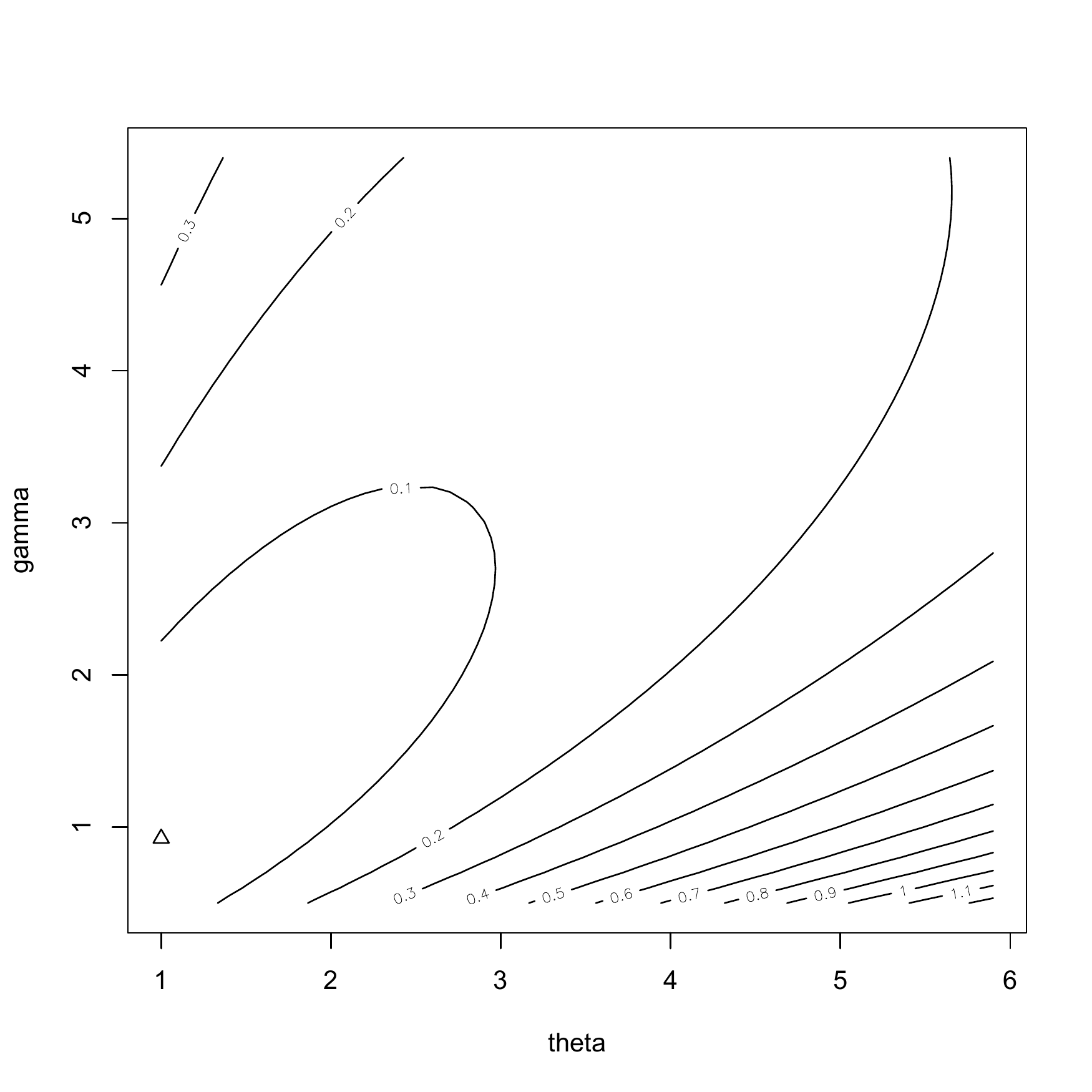}
\caption{\label{ex21}
Contour plot for $\rho_{\gamma,\theta}$ in Example~\ref{ex:pois-geom}. The triangle point indicates the minimum.}
\end{center}
\end{figure}

\begin{figure}[ht]\begin{center}
\includegraphics[height=3in,width=3in]{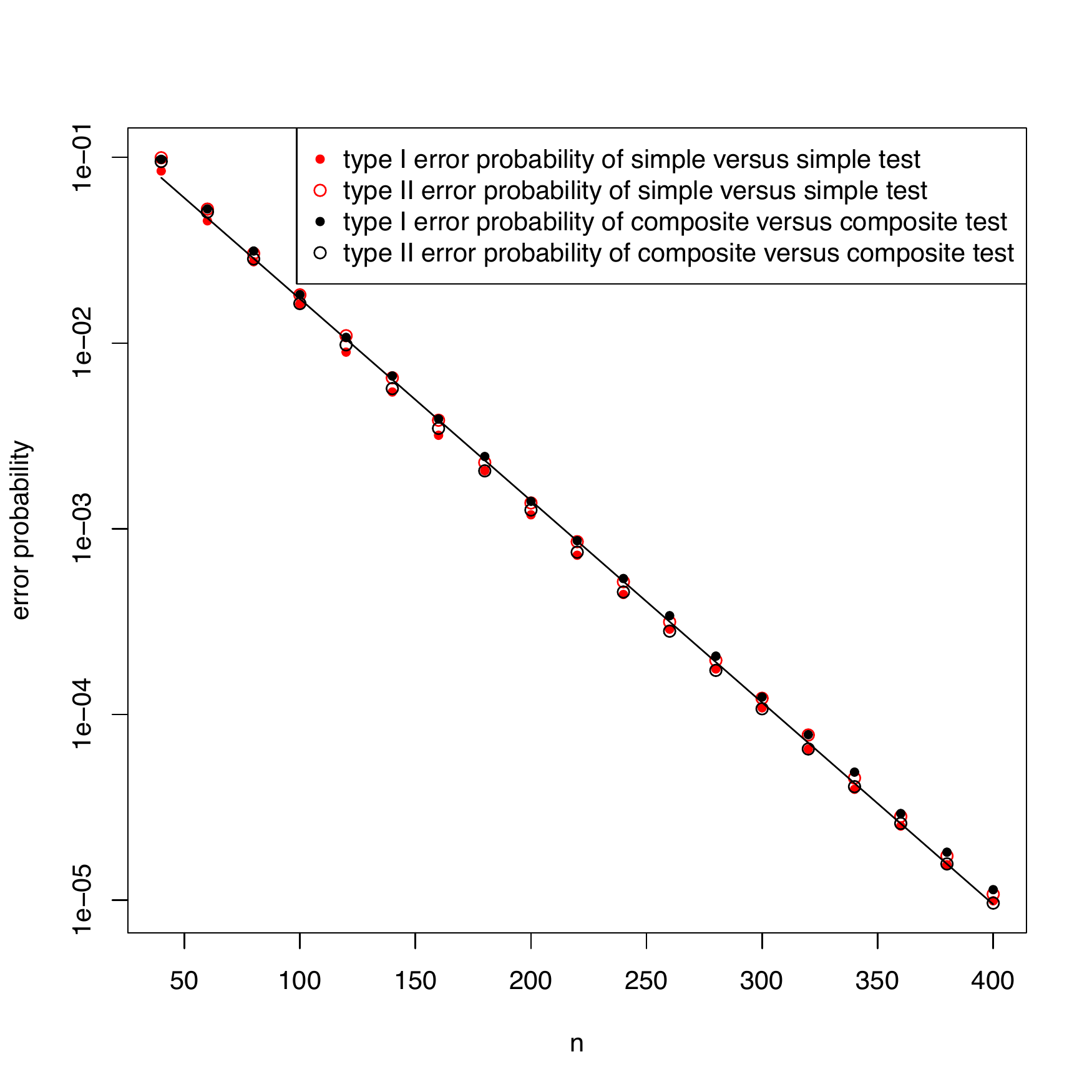}
\caption{\label{ex22}
Maximal type I and type II error probabilities ($y$-coordinate) as a function of sample size ($x$-coordinate) in Example~\ref{ex:pois-geom}.}
\end{center}
\end{figure}



%

\begin{figure}
  \begin{subfigure}[b]{0.45\textwidth}
    \includegraphics[width=\textwidth]{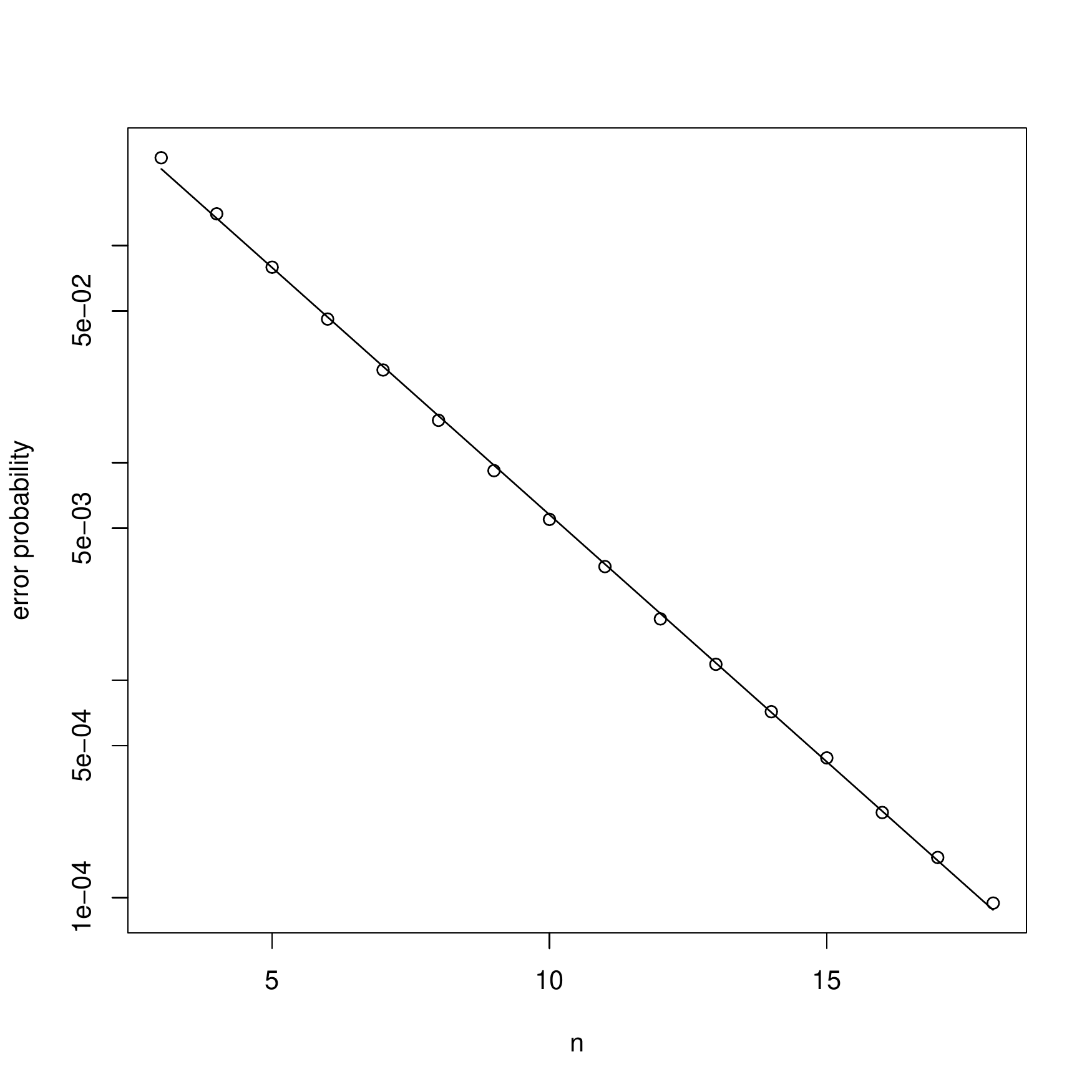}
    \caption{}
    \label{fig:linear_error}
  \end{subfigure}
  \begin{subfigure}[b]{0.45\textwidth}
    \includegraphics[width=\textwidth]{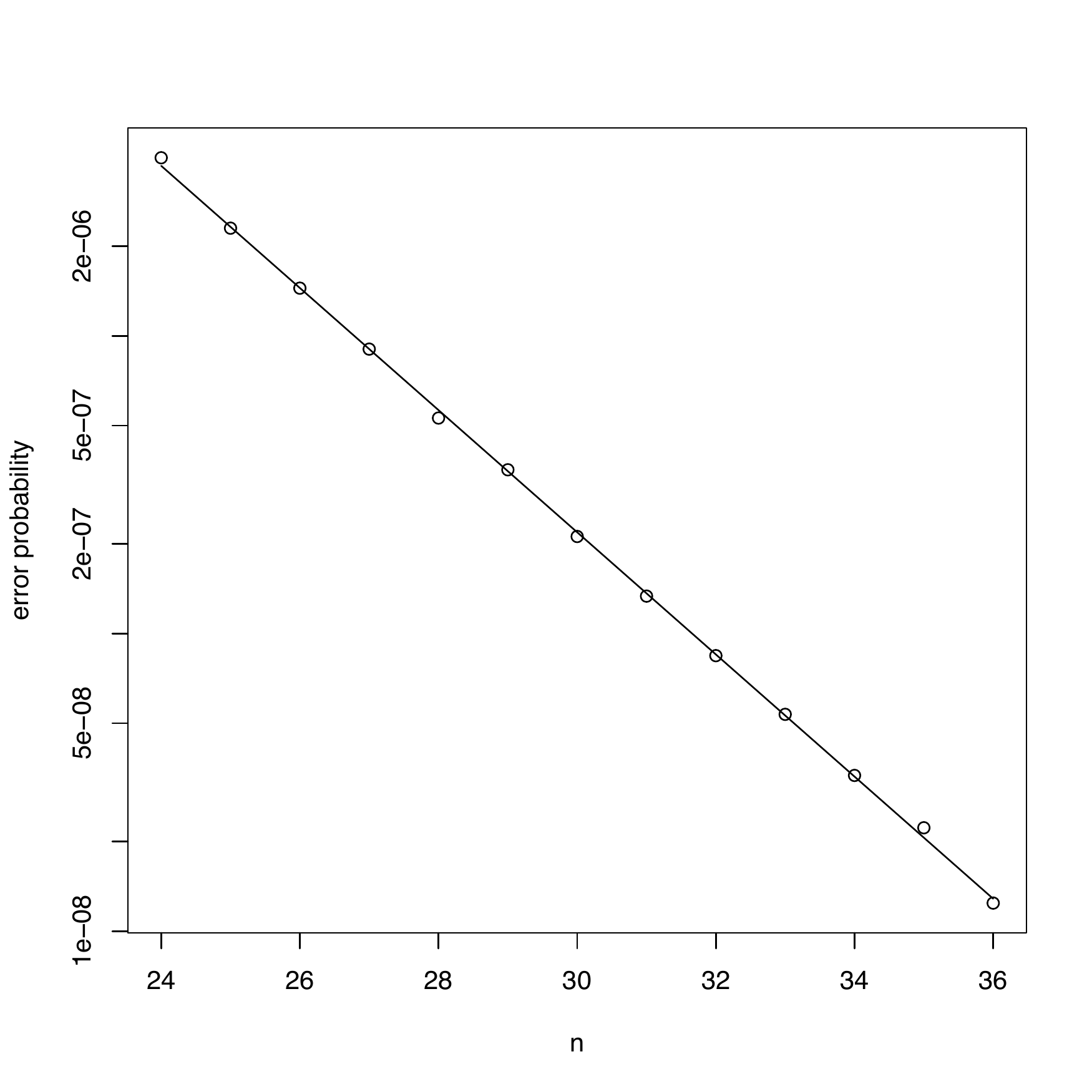}
    \caption{}
    \label{fig:linear_small}
  \end{subfigure}
  \caption{Error probability($y$-coordinate) in Example~\ref{ex:linear} as a function of sample size($x$-coordinate).}
\end{figure}

\section{Concluding remarks}\label{SecConclude}

The generalized likelihood ratio test of separate parametric
families that was put forth by Cox in his two seminal papers has
received a great deal of attention in the statistics and
econometrics literature. The present investigation takes the
viewpoint of an early work by Chernoff (1952) where testing a
simple null versus a simple alternative is considered. By imposing
that the two types of error probabilities decay at the same rate, we
extend the Chernoff index to the case of the Cox test.

Our results are under the basic assumption that the data come from
one of the parametric families under consideration. It is often
the case that none is the true model. It would be of interest to
formulate error probabilities for this case and to see if
similar exponential decay results continue to hold.

An initial motivation that led to the Cox formulation of the
problem comes from the survival analysis where different models are
used to fit failure time data. The econometrics literature also
contains much subsequent development. Semiparametric models that
contain infinite dimensional nuisance parameters are widely used
in both econometrics and survival analysis. It would be of
interest to develop parallel results for testing separate
semiparametric models.

\appendix

\section{Proof of Lemma \ref{ThmSimpleComposite}}

Throughout the proof, we adopt the following notation $a_n \cong b_n$ if $\log a_n \sim \log b_n$.
We define the log-likelihood ratio as
\begin{equation*}
l_\gamma (x) = \log h_\gamma(x) - \log g(x).
\end{equation*}
The generalized log-likelihood ratio statistic is defined as
$$l=\sup_\gamma \sum_{i=1}^n l^i_\gamma$$
where $l^i_\gamma = l_\gamma(X_i)$. The generalized likelihood ratio test admits the rejection region
$$C_\lambda = \{e^l>\lambda\}.$$
We consider the case that $\lambda =1$ and show that for this particular choice of $\lambda$ the maximal type I and type II error probabilities decay exponentially fast with the same rate. We let $\gamma_*=\arg\inf\rho_\gamma$ and thus $\rho = \rho_{\gamma_*}$.

Based on Chernoff's calculation of large deviations for the log-likelihood ratio statistic, we proceed to the calculation of  the type I error probability
$$\Pr _g(l>0) = \Pr _g\Big(\sup_\gamma\sum_{i=1}^n l^i_\gamma>0\Big). 
$$
We now provide an approximation of the right-hand side, which requires a lower bound and an upper bound.
We start with the lower bound by noticing that
\begin{equation}\label{lb1}
\Pr _g\Big(\sup_\gamma\sum_{i=1}^n l^i_\gamma>0\Big)\geq \sup_\gamma \Pr _g\Big( \sum_{i=1}^n l^i_\gamma>0\Big)
\end{equation}
that is a simple lower bound.
According to Proposition \ref{PropChernoff}, the right-hand side is bounded from below by $$\geq e^{- \{1+o(1)\}n \rho}$$
where $\rho=\min\rho_\gamma.$
For the upper bound and with some $\beta>0$, we split the probability
\begin{eqnarray}\label{split}
\Pr _g\Big(\sup_\gamma\sum_{i=1}^n l^i_\gamma>0\Big)&\leq & \Pr _g\Big(\sup_\gamma\sum_{i=1}^n l^i_\gamma>0,  \sup_\gamma \Big |\sum_{i=1}^n \nabla l^i_\gamma\Big| < e^{n^{1-\beta}} \Big) \notag\\
&&~~~+ \Pr _g\Big ( \sup_\gamma \Big |\sum_{i=1}^n \nabla l^i_\gamma\Big| \geq e^{n^{1-\beta}} \Big).
\end{eqnarray}
The first term on the right-hand side  is bounded by Lemma \ref{LemRF}.

\begin{lemma}\label{LemRF}
Consider a random function $\eta_n(\theta)$ living on a $d$-dimensional compact domain $\theta \in D$, where $n$ is an asymptotic parameter that will be send to infinity. Suppose that $\eta_n(\theta)$ is almost surely differentiable with respect to $\theta$ and for each $\theta$, there exists a rate $\rho(\theta)$ such that
$$\Pr \{\eta_n(\theta) > \zeta_n \} \cong e^{-n \rho(\theta)} \quad \mbox{for all $\zeta_n /n\to 0$ as $n\to \infty$}$$
where the above convergence is uniform in $\theta$.
Then, we have the following approximation
$$\liminf_{n \to \infty} -\frac 1 n \log \Pr \{\sup_{\theta \in D} \eta_n (\theta) >0, \sup_{\theta \in D}|\nabla \eta_n(\theta)| < e^{n^{1-\beta}}\}   \geq \min_\theta \rho(\theta)$$
for all $\beta >0$.
\end{lemma}
With the aid of Proposition~\ref{PropChernoff}, we have that the random function $\sum_{i=1}^n l_{\gamma}^i$ satisfies the assumption in Lemma~\ref{LemRF} with $\rho(\gamma)= \rho_{\theta\gamma}.$
Then the first term in \eqref{split} is bounded from the above by $e^{- \{1+o(1)\}n \rho}.$
For the second term in \eqref{split}, according to condition A3, we choose $\beta$ sufficiently small such that
$$\Pr _g\Big ( \sup_\gamma \Big |\sum_{i=1}^n \nabla l^i_\gamma\Big| \geq
e^{n^{1-\beta}} \Big)\leq n \times \Pr _g(\sup_\gamma |\nabla l^i_\gamma|> n^{-1}e^{n^{1-\beta}}) = o(e^{-n \rho}).$$
Thus, we obtain an upper bound
$$\Pr _g\Big(\sup_\gamma\sum_{i=1}^n l^i_\gamma>0\Big) \leq e^{-n \{\rho +o(1)\}}
.
$$
Then, the type I error probability is approximated by
\begin{equation}\label{type1}
e^{-n \rho}\cong \sup_\gamma \Pr _g\Big( \sum_{i=1}^n l^i_\gamma>0\Big)\leq \Pr _g\Big(\sup_\gamma\sum_{i=1}^n l^i_\gamma>0\Big) \leq e^{-n \{\rho +o(1)\}}.
\end{equation}

We now consider the type II error probability
$\alpha_2 = \sup_\gamma \Pr _{h_\gamma}( l < 0 ).$
For each $\gamma$, note that
$$\Pr _{h_\gamma}( l < 0 ) = \Pr _{h_\gamma}\Big(\sup_{\gamma_1} \sum_{i=1}^n l^i_{\gamma_1}<0\Big)\leq  \Pr _{h_\gamma}\Big(\sum_{i=1}^n l^i_{\gamma}<0\Big).$$
Note that the right-hand side is the type II error probability of
the likelihood ratio test. According to Chernoff's calculation, we
have that
$$\Pr _{h_\gamma}( l < 0 ) \leq \Pr _{h_\gamma}\Big(\sum_{i=1}^n l^i_{\gamma}<0\Big)\cong e^{-n\rho_\gamma}$$
for all $\gamma$. We take maximum with respect to $\gamma$ on both sides  and obtain that
\begin{equation}\label{ub}
\sup_\gamma\Pr _{h_\gamma}( l < 0 )\leq \sup _\gamma \Pr _g\Big( \sum_{i=1}^n l^i_{\gamma}>0\Big)\cong  e^{-n \min _\gamma \rho_{\gamma}}.
\end{equation}
Thus, the maximal type II error probability has an asymptotic upper bound that decays at the rate of the Chernoff index.

In what follows, we show that this asymptotic upper bound is asymptotically achieved.
We choose $\lambda_n$ possibly depending on $g$ such that 
$$\Pr _g\Big(\sup_\gamma \sum_{i=1}^n l^i_\gamma>0\Big) = \Pr _g\Big( \sum_{i=1}^n l^i_{\gamma_*}>n\lambda_n\Big).$$
Note that $g$ is fixed and the probabilities on both sides of the above identity decay at the rate $e^{-n\rho}$. Together with the continuity of the large deviations rate function, it must be true that $\lambda_n \to 0-$. We apply Neyman-Pearson lemma to the simple null $H_0: f=g$ versus simple alternative $H_1:f=h_{\gamma_*}$. Note that $\{\sum_{i=1}^n l^i_{\gamma_*}>n\lambda_n\}$ is a uniformly most powerful test and $\{\sup_\gamma \sum_{i=1}^n l^i_\gamma>0\}$ is a test with the same type I error probability. Then, we have that
\begin{equation}\label{lbb}
\Pr _{h_{\gamma_*}}\Big(\sup_\gamma \sum_{i=1}^n l^i_\gamma<0\Big) \geq \Pr _{h_{\gamma_*}}\Big( \sum_{i=1}^n l^i_{\gamma_*}<n\lambda_n\Big).
\end{equation}
That is, the type II error probability of the generalized likelihood ratio test must be greater than that of the likelihood ratio test under the simple alternative $h_{\gamma_*}$.
Note that $\lambda_n \to 0-$. Thanks to the  the continuity of the large deviations rate function, we have that
\begin{equation}\label{lb}
\Pr _{h_{\gamma_*}}\Big( \sum_{i=1}^n l^i_{\gamma_*}<n\lambda_n\Big)\cong \Pr _{h_{\gamma_*}}\Big( \sum_{i=1}^n l^i_{\gamma_*}<0\Big)\cong e^{-n \rho}.
\end{equation}
 Put together \eqref{ub}, \eqref{lbb}, and \eqref{lb}, we have that
$$\sup_\gamma\Pr _{h_\gamma}( l < 0 ) \cong e^{- n \rho}.$$
Thus, we conclude the proof.


\section{Proof of Theorem \ref{ThmCC}}

The one-to-one log-likelihood ratio is
$$l_{\theta \gamma} (x) = \log h_\gamma(x) - \log g_\theta(x).$$
The generalized log-likelihood ratio statistic is
$$l=\sup_\gamma \sum_{i=1}^n\log h_\gamma(X_i)- \sup_\theta \sum_{i=1}^n\log g_\theta(X_i)   = \inf_\theta \sup_\gamma \sum_{i=1}^n l^i_{\theta\gamma}$$
and the rejection region is
$$C_\lambda = \{e^l> \lambda\}.$$
We define that $\gamma(\theta) = \arg\inf_{\gamma} \rho_{\theta\gamma}$, and $\theta(\gamma) = \arg\inf_{\theta} \rho_{\theta\gamma}$,
and $(\theta_*,\gamma_*) = \arg\inf_{\theta,\gamma} \rho_{\theta\gamma}.$
Note that the null and the alternative are now symmetric, thus we only need to consider one of the two types of error probabilities. We consider the type II error probability.
We now define
$$k_\theta =\sup_{\gamma} \sum_{i=1}^n l^i_{\theta\gamma}.$$
For each given $\theta$ and $\gamma$, we have a simple upper bound
\begin{equation}\label{eqst}
 \Pr _{h_\gamma}(k_\theta<0)\leq  \Pr _{h_\gamma}\Big(\sum_{i=1}^n l^i_{\theta\gamma}<0\Big)\cong  e^{-n \rho_{\theta\gamma}}.
 \end{equation}
We now proceed to the type II error probability if $h_\gamma$ is the true distribution, that is
$$\Pr _{h_\gamma}(\inf_\theta k_\theta<0)\leq \Pr _{h_\gamma}(\inf_\theta k_\theta<0; \sup_\theta |\nabla k_\theta|< e^{n^{1-\beta}}) + \Pr _{h_\gamma}(\sup_\theta |\nabla k_\theta|\geq e^{n^{1-\beta}}).$$
The first term on the right-hand-side is bounded by Lemma \ref{LemRF} combined with \eqref{eqst}
$$\Pr _{h_\gamma}(\inf_\theta k_\theta<0; \sup_\theta |\nabla k_\theta|< e^{n^{1-\beta}})\leq e^{-n \{\inf_\theta \rho_{\theta\gamma}+o(1)\}}.$$
For the second term, we have that
\begin{eqnarray*}
\Pr _{h_\gamma}\{\sup_\theta |\nabla (\sup_{\gamma} \sum_{i=1}^n l^i_{\theta\gamma})|\geq e^{n^{1-\beta}}\}
&\leq& \Pr _{h_\gamma}(\sup_\theta \sup_{\gamma} \sum_{i=1}^n |\nabla l^i_{\theta\gamma}|\geq e^{n^{1-\beta}})\\
&\leq& n \Pr _{h_\gamma}(\sup_\theta \sup_{\gamma}  |\nabla l^i_{\theta\gamma}|\geq n^{-1}e^{n^{1-\beta}})= o(e^{-n \rho}).
\end{eqnarray*}
Thus,  we have that
$$\Pr _{h_\gamma}(\inf_\theta k_\theta<0) = \Pr _{h_\gamma} (l<0) \leq e^{-n \{\inf_\theta \rho_{\theta\gamma}+o(1)\}}, $$
which provides an upper bound for the type II error probability
$$\sup _\gamma \Pr _{h_\gamma}(l<0) \leq  e^{-n \{\inf_{\theta,\gamma} \rho_{\theta\gamma}+o(1)\}}.$$
We now provide a lower bound.
For a given $\theta$ and $\gamma(\theta)=\arg\inf_{\gamma}\rho_{\theta\gamma}$, applying proof of  Lemma \ref{ThmSimpleComposite} for the type II error probability by considering $H_0: f=g_\theta$ and $H_1: f\in \{h_\gamma: \gamma \in \Gamma\}$, we have that
$$
\Pr _{h_{\gamma(\theta)}}(k_\theta<0) \cong e^{-n \rho_{\theta \gamma(\theta)}}.
$$
and thus
$$\Pr _{h_{\gamma(\theta)}}(\inf _\theta k_\theta<0)\geq \Pr _{h_{\gamma(\theta)}}(k_\theta<0)  \cong  e^{-n  \rho_{\theta \gamma(\theta)}}.$$
We set $\theta = \theta_*$ in the above asymptotic identity and  conclude the proof.

\section{Proof of Lemma \ref{LemRF}}

We consider a change of measure on the continuous sample path space $Q_\zeta$ that admits the following Radon-Nikodym derivative
\begin{equation}\label{IS}
\frac{dQ_\zeta}{dP} = \frac{mes(A_\zeta)}{ \int_D
P\{\eta_n(\theta)>\zeta\}d\theta},
\end{equation}
where $A_\zeta = \{\theta\in D: \eta_n(\theta) >\zeta\}$ and $mes(\cdot)$ is the Lebesgue measure. Throughout the proof, we choose $\zeta = -1$.
To better understand the measure $Q_\zeta$, we provide another description of the sample path generation of $\eta_n$  from $Q_\zeta$, that requires the following three steps
\begin{itemize}
\item[1.] Sample a random index $\tau\in D$ following the density function
$$h(\tau) = \frac{\Pr \{\eta_n(\tau) > \zeta\}}{\int_D \Pr \{\eta_n(\theta) > \zeta\}d\theta};$$
\item[2.] Sample $\eta_n(\tau)$ given that $\eta_n(\tau)>\zeta$;
\item[3.] Sample $\{\eta_n(\theta): \theta \neq \tau\}$ from the original conditional distribution given the realized value of $\eta_n(\tau)$, that is, $\Pr \{\cdot |\eta_n(\tau)\}$.
\end{itemize}
To verify that the measure induced by the above sampling procedure is the same as that given by \eqref{IS}, see \cite{ABL09} that provides a discrete analogue of the above change of measure.

With these constructions, the interesting probability is given by
\begin{eqnarray*}
&&\Pr \{\sup_{\theta \in D} \eta_n (\theta) >0, \sup_{\theta \in D}|\nabla \eta_n(\theta)| <e^{n^{1-\beta}}\}\\
&=& E^{Q_\zeta}\Big\{\frac{dP}{dQ_\zeta}; \sup_{\theta \in D} \eta_n (\theta) >0, \sup_{\theta \in D}|\nabla \eta_n(\theta)| < e^{n^{1-\beta}}\Big\}\\
&=& E^{Q_\zeta}\Big\{\frac{1}{mes(A_\zeta)}; \sup_{\theta \in D} \eta_n (\theta) >0, \sup_{\theta \in D}|\nabla \eta_n(\theta)| < e^{n^{1-\beta}}\Big\}\\
&&~~~~~~\times \int_D \Pr (\eta_n(\theta) > \zeta)d\theta
\end{eqnarray*}
Via the condition of this lemma, we have that
$$\int_D \Pr (\eta_n(\theta) > \zeta)d\theta  \cong e^{-n \min _\theta \rho(\theta)}.$$
Thus, it is sufficient to show that $$E^{Q_\zeta}\Big\{\frac{1}{mes(A_\zeta)}; \sup_{\theta \in D} \eta_n (\theta) >0, \sup_{\theta \in D}|\nabla \eta_n(\theta)| < e^{n^{1-\beta}}\Big\}$$ cannot be too large.
On the set $\{\sup_{\theta \in D} \eta_n (\theta) >0, \sup_{\theta \in D}|\nabla \eta_n(\theta)| < n^{1-\beta}\}$, the volume $mes(A_\zeta)$ is in fact lower bounded. Let $\theta_*$ be the maximizer of $\eta_n(\theta)$ and thus $\eta_n(\theta_*) > 0$. On the other hand, the gradient of $\eta_n$ is upper bounded by $e^{n^{1-\beta}}$. Therefore, there exists  a small region of radius $e^{-n^{1-\beta}}$ in which $\eta_n$ will be above $\zeta=-1$. Thus, $mes(A_\zeta)$ is lower bounded by $\varepsilon_0 e^{-dn^{1-\beta}}.$
Thus, the bound
$$\Pr (\sup_{\theta \in D} \eta_n (\theta) >0, \sup_{\theta \in D}|\nabla \eta_n(\theta)| < n^\beta)\leq \frac{ e^{dn^{1-\beta}}}{\varepsilon_0} \int_D \Pr (\eta_n(\theta) > \zeta)d\theta \cong e^{-n \min _\theta \rho(\theta)}$$
concludes the proof.

\section{Proof of Corollary~\ref{cor:non-compact}}

The proof is very similar to that of Theorem \ref{ThmCC} and therefore we omit some repetitive steps.
We first consider the type I error probability,
$$
\sup_{\theta\in\Theta}P_{g_{\theta}}(LR_n>1).
$$
For each $\theta\in\Theta$, we establish an upper bound for 
\begin{equation}\label{eq:type-I-prob}
P_{g_{\theta}}(LR_n>1)= P_{g_{\theta}}\Big(\sup_{\gamma\in\Gamma}\log h_{\gamma}(X_i)-\sup_{\theta\in \Theta}\log g_{\theta}(X_i)>0\Big).
\end{equation}
The event $$\{\sup_{\gamma\in\Gamma}\log h_{\gamma}(X_i)-\sup_{\theta\in \Theta}\log g_{\theta}(X_i)>0\}$$ implies
$$
\{\sup_{\gamma\in\Gamma}\log h_{\gamma}(X_i)-\log g_{\theta}(X_i)>0\}.
$$
Thus, we have 
$$
P_{g_{\theta}}(LR_n>1)\leq P_{g_{\theta}}\Big(\sup_{\gamma\in\Gamma}\log h_{\gamma}(X_i)-\log g_{\theta}(X_i)>0\Big).
$$
We split the probability
\begin{eqnarray}
&& P_{g_{\theta}}\Big(\sup_{\gamma\in \Gamma}\log h_{\gamma}(X_i)-\log g_{\theta}(X_i)>0\Big)\notag\\
&\leq& P_{g_{\theta}}\Big(\sup_{\gamma\in \Gamma}\log h_{\gamma}(X_i)-\log g_{\theta}(X_i)>0,~\hat \gamma\in A_{\theta}\Big)+ P_{g_{\theta}}\Big(\hat \gamma\in A_{\theta}^c\Big)\notag\\
&\leq & P_{g_{\theta}}\Big(\sup_{\gamma\in A_{\theta}}\log h_{\gamma}(X_i)-\log g_{\theta}(X_i)>0\Big)+ P_{g_{\theta}}\Big(\hat \gamma\in A_{\theta}^c\Big).\label{eq:upper-bound-inequality}
\end{eqnarray}
According to Assumption A4, the second term is $o(e^{-n\rho})$.
For the first term, notice that $A_{\theta}$ is a compact subset of $R^{d_g}$. The conditions for Lemma~\ref{ThmSimpleComposite} are satisfied. According to Lemma~\ref{ThmSimpleComposite}, the first term in \eqref{eq:upper-bound-inequality} is bounded above by
$$
e^{-(1+o(1))n\times \min_{\gamma\in A_{\theta}}\rho_{\theta\gamma}}\leq e^{-(1+o(1))n\times \min_{\gamma\in \Gamma}\rho_{\theta\gamma}}\leq e^{-(1+o(1))n\times \min_{\theta,\gamma}\rho_{\theta\gamma}}.
$$
Combining the upper bounds for the first and second terms in \eqref{eq:upper-bound-inequality}, we have
$$
P_{g_{\theta}}(LR_n>1)\leq e^{-(1+o(1))n\times \min_{\theta,\gamma}\rho_{\theta\gamma}}.
$$
The above derivation is uniform in $\theta$. We obtain an upper bound for the type I error
$$
\sup_{\theta}P_{g_{\theta}}(LR_n>1)\leq e^{-(1+o(1))n\times \min_{\theta,\gamma}\rho_{\theta\gamma}}.
$$
Similarly, we obtain an upper bound for the type II error probability
$$
\sup_{\gamma}P_{h_{\gamma}}(LR_n\leq1)\leq e^{-(1+o(1))n\times \min_{\theta,\gamma}\rho_{\theta\gamma}}.
$$
Now we proceed to a lower bound for the type I error probability. Upon having the upper bounds for both type I and type II error probabilities, the lower bounds for type I and type II error probabilities can be derived using the same argument as that in the proof of Theorem~\ref{ThmCC}. We omit the details.

\bigskip

\section{Proof of Theorem \ref{thm:error-under-theta0}}

The proof of the theorem consists of establishing upper and lower bounds for the probability
$$
P_{g_{\theta_0}}(LR_n>e^{nb})=P_{g_{\theta_0}}\Big(\sup_{\gamma\in\Gamma}\inf_{\theta\in\Theta}\sum_{i=1}^n [\log h_{\gamma}(X_i)-\log g_{\theta}(X_i)]> nb\Big).
$$
\paragraph{Upper bound}
The event 
$$\Big\{\sup_{\gamma\in\Gamma}\inf_{\theta\in\Theta}\sum_{i=1}^n \log h_{\gamma}(X_i)-\log g_{\theta}(X_i)> nb\Big\}$$ 
implies $$\Big\{\sup_{\gamma\in\Gamma}\sum_{i=1}^n \log h_{\gamma}(X_i)-\log g_{\theta^{\dagger}}(X_i) >nb\Big\}.$$
Therefore, we have an upper bound
\begin{equation}\label{eq:upper-bound-first}
P_{g_{\theta_0}}(LR_n>1)\leq P_{g_{\theta_0}}\Big(
\sup_{\gamma} \sum_{i=1}^n \log h_{\gamma}(X_i)-\log g_{\theta^\dagger}(X_i)
> nb\Big).
\end{equation}
We split the probability
\begin{eqnarray}\label{eq:splitprob}
& & P_{g_{\theta_0}}\Big(
\sup_{\gamma} \sum_{i=1}^n [\log h_{\gamma}(X_i)-\log g_{\theta^\dagger}(X_i)]
>nb\Big)\\
&\leq& P_{g_{\theta_0}}\Big(\sup_{\gamma} \sum_{i=1}^n [\log h_{\gamma}(X_i)-\log g_{\theta^\dagger}(X_i)]
>nb, \sup_{\gamma}\Big|\sum_{i=1}^n \nabla_{\gamma} \log h_{\gamma}(X_i)\Big|<e^{n^{1-\beta}}
\Big)\notag\\
&&+ P_{g_{\theta_0}}\Big(\sup_{\gamma}\sum_{i=1}^n\Big|\nabla_{\gamma} \log h_{\gamma}(X_i)\Big|\geq e^{n^{1-\beta}}
\Big).\notag
\end{eqnarray}
We establish upper bounds of the first and second terms in \eqref{eq:splitprob} separately. 
For the first term, let $\eta_n(\gamma)= \sum_{i=1}^n [\log h_{\gamma}(X_i)-\log g_{\theta^\dagger}(X_i)]-nb$. For each $\gamma$, the exponential decay rate of the probability
\begin{equation}\label{eq: eta gamma decay rate}
\log P_{g_{\theta_0}}(\eta_n(\gamma)\geq 0)\leq n\log \inf_{\lambda} M_{g_{\theta_0}}(\lambda,\gamma,\td).
\end{equation}
is established through standard large deviation calculation.
Thanks to Lemma~\ref{LemRF}
and \eqref{eq: eta gamma decay rate},  the first term in \eqref{eq:splitprob} is bounded above by
$$
 \sup_{\gamma}\inf_{\lambda} \{M_{g_{\theta_0}}(\td,\lambda,\gamma)\}^{(1+o(1))n}= e^{-(1+o(1))n \rd}.
$$
For the second term, according to the Assumption A3,
$$
P_{g_{\theta_0}}\Big(\sup_{\gamma}\sum_{i=1}^n\Big|\nabla_{\gamma} \log h_{\gamma}(X_i)\Big|\geq e^{n^{1-\beta}}
\Big)\leq n\times P_{g_{\theta_0}}(\sup_{\gamma}|\nabla_{\gamma}\log h_{\gamma}(X_i)|>n^{-1}e^{n^{1-\beta}})=o(e^{-n\rd}).
$$
Combining the analyses for both the first and the second term, we arrive at an upper bound
$$
P_{g_{\theta_0}}(LR_n>e^{nb})\leq e^{-(1+o(1))n\rd}.
$$
\paragraph{Lower bound} 
Recall that
$$
\frac{dQ^{\dagger}}{dP_{g_{\theta_0}}}=\exp\Big\{\lambda^{\dagger}(\log h_{\gamma^{\dagger}}(X)-\log g_{\theta^{\dagger}}(X))-nb\Big\}/M^{\dagger}_{g_{\theta_0}}.
$$
Then, the probability can be written as
$$
P_{g_{\theta_0}}(LR_n>e^{nb})=E^{\Qd}\Big\{
\frac{dP_{g_{\theta_0}}}
{d\Qd}; \sum_{i=1}^n [\log h_{\hat\gamma}(X_i)-\log g_{\hat\theta}(X_i)]> nb
\Big\},
$$
where $\hat\gamma$ and $\hat\theta$ are the maximum likelihood estimators for the $h$-family and the $g$-family respectively.
According to the definition of $\Qd$, the above display is equal to 
\begin{equation}\label{eq: change of measure for lower bound}
e^{-n\rd}E^{\Qd}\Big\{e^{-\ld [\sum_{i=1}^n \log h_{\gd}(X_i)-\log g_{\td}(X_i)-nb]};  \sum_{i=1}^n \log h_{\hat\gamma}(X_i)-\log g_{\hat\theta}(X_i)>nb\Big\},
\end{equation}
where $\rd=-\log M^{\dagger}_{g_{\theta_0}}$.
We now establish a lower bound for 
$$
I\triangleq E^{\Qd}\Big\{e^{-\ld [\sum_{i=1}^n \log h_{\gd}(X_i)-\log g_{\td}(X_i)-nb]};  \sum_{i=1}^n \log h_{\hat\gamma}(X_i)-\log g_{\hat\theta}(X_i)>nb\Big\}.
$$
Because $e^{-\ld [\sum_{i=1}^n \log h_{\gd}(X_i)-\log g_{\td}(X_i)-nb]}$ is  positive, we have
\begin{equation}\label{eq:D-I}
 I 
\geq  E^{\Qd}\Big\{e^{-\ld [\sum_{i=1}^n \log h_{\gd}(X_i)-\log g_{\td}(X_i)-nb]};  \sum_{i=1}^n \log h_{\hat\gamma}(X_i)-\log g_{\hat\theta}(X_i)>nb, E_1\Big\},
\end{equation}
where
$$E_1=\Big\{\Big|\sum_{i=1}^n \log h_{\gd}(X_i)-\log g_{\td}(X_i)-nb\Big|\leq \sqrt{n}\Big|\Big\}.$$
On the set $E_1$, we have the following inequality of the integrand 
$$e^{-\ld [\sum_{i=1}^n \log h_{\gd}(X_i)-\log g_{\td}(X_i)-nb]}\geq e^{-|\ld|\sqrt{n}}. $$ 
We plug the above inequality back to \eqref{eq:D-I} and obtain a lower bound for
\begin{equation}\label{eq: I lower}
I\geq e^{-|\ld|\sqrt{n}} \Qd\Big( \{\sum_{i=1}^n \log h_{\hat\gamma}(X_i)-\log g_{\hat\theta}(X_i)>nb\}\cap E_1\Big).
\end{equation}
For the rest of the proof, we develop a lower bound for the probability $$\Qd\Big( \{\sum_{i=1}^n \log h_{\hat\gamma}(X_i)-\log g_{\hat\theta}(X_i)>nb\}\cap E_1\Big).$$
The maximum likelihood estimator $\hat{\gamma}$ satisfies the inequality
\begin{equation}\label{eq:mle-inequality}
\sum_{i=1}^n \{\log h_{\hat{\gamma}}(X_i)-\log h_{\gd}(X_i)\}\geq 0.
\end{equation}
Furthermore, with the aid of Rolle's Theorem, there exists $\tilde{\theta}$  such that
\begin{eqnarray}
&&\sum_{i=1}^n \{\log g_{\hat\theta}(X_i)-\log g_{\td}(X_i)\}\notag\\
&=& (\hat{\theta}-\td)\cdot\sum_{i=1}^n\nabla_{\theta} \log g_{\td}(X_i)
 + \frac{1}{2}(\hat{\theta}-\td)^\top\sum_{i=1}^n\nabla^2_{\theta}  g_{\tilde{\theta}}(X_i)(\hat{\theta}-\td)
,\label{eq:rolle}
\end{eqnarray}
where ``$\nabla^2_{\theta}$''  denotes the Hessian matrices with respect to $\theta$  and ``$\cdot$'' denotes the inner product between vectors.
\eqref{eq:mle-inequality} and \eqref{eq:rolle} together give
\begin{eqnarray}
&&\sum_{i=1}^n \{\log h_{\hat{\gamma}}(X_i)-\log g_{\hat{\theta}}(X_i)\}-\sum_{i=1}^n \{\log h_{\gd}(X_i)-g_{\td}(X_i)\} \notag\\
&\geq& -(\hat{\theta}-\td)\cdot\sum_{i=1}^n\nabla_{\theta} \log g_{\td}(X_i)
 - \frac{1}{2}(\hat{\theta}-\td)^\top\sum_{i=1}^n\nabla^2_{\theta}  g_{\tilde{\theta}}(X_i)(\hat{\theta}-\td).\label{eq:difference}
 \end{eqnarray}

We define
$$
E_2=\Big\{(\hat{\theta}-\td)^\top\sum_{i=1}^n\nabla_{\theta} \log g_{\td}(X_i)\leq\frac{\sqrt{n}}{4}\Big \},
$$
$$
E_3=\Big\{\frac{1}{2}|\hat{\theta}-\td|^2 \sup_{\theta}\sum_{i=1}^n |\nabla^2_{\theta} \log g_{\theta}(X_i)|\leq \frac{\sqrt{n}}{4}\Big\},
$$
$$
E_4= \Big\{\frac{\sqrt{n}}{2}<\sum_{i=1}^n [\log h_{\gd}(X_i)-\log g_{\td}(X_i)]-nb\leq \sqrt{n}\Big\}.
$$
Based on \eqref{eq:difference}, we have that
$$
(E_2\cap E_3 \cap E_4) \subset \{\sum_{i=1}^n \log h_{\hat\gamma}(X_i)-\log g_{\hat\theta}(X_i)>nb\}\cap E_1.
$$
We insert this to \eqref{eq:D-I}, and obtain that
\begin{equation}\label{eq:I-lower2}
I\geq e^{-|\ld|\sqrt{n}}\Qd(E_2\cap E_3\cap E_4)\geq e^{-|\ld|\sqrt{n}}\Big\{\Qd(E_4)- \Qd(E_2^c) - \Qd(E_3^c)\Big\}.
\end{equation}
For the rest of the proof, we develop  upper bounds for $\Qd(E_2^c)$ and $\Qd(E_3^c)$ and a lower bound for $\Qd(E_4)$. 
For $\Qd(E_4)$, because $\ld=\arg\inf_{\lambda}M_{g_{\theta_0}}(\td,\gd,\lambda)$, we have 
$$
\frac{\partial }{\partial \lambda}M_{g_{\theta_0}}(\td,\gd,\ld)=0.
$$
Consequently,
$$E^{\Qd}(\log h_{\gd}(X)-\log g_{\td}(X)-b)
= (M^{\dagger}_{g_{\theta_0}})^{-1}\frac{\partial }{\partial \lambda}M_{g_{\theta_0}}(\td,\gd,\ld)=0. 
$$ 
According to the central limit theorem, there exists $\varepsilon_0>0$ such that
$$\lim\inf_{n\to\infty} \Qd(E_4)>\varepsilon_0.$$
Thus a lower bound for $\Qd(E_4)$ has been derived.
Before we proceed to upper bounds for $\Qd(E_2^c)$ and $\Qd(E_3^c)$, we establish the following lemma, whose proof is provided in Appendix~\ref{sec:proof-for-lemma}.
\begin{lemma}\label{lemma:consistent}
Under the settings of Theorem~\ref{thm:error-under-theta0}, we have
$$
\gd=\bar{\gamma}\mbox{ and } \td=\bar \theta.
$$
\end{lemma}
We now proceed to an upper bound of $\Qd(E_2^c)$.
We split the sum
\begin{eqnarray}\label{eq:minus-mean}
&&(\hat{\theta}-\td)^\top\sum_{i=1}^n\nabla_{\theta} \log g_{\td}(X_i)\\
&&~~~~=  (\hat{\theta}-\td)^\top \sum_{i=1}^n[\nabla_{\theta} \log g_{\td}(X_i) - E^{\Qd}\nabla_{\theta}g_{\td}(X_i)] +n(\hat{\theta}-\td)^\top E^{\Qd}\nabla_{\theta}g_{\td}(X)\notag
\end{eqnarray}
Note that $\hat{\theta}\in T_{\td}\Theta$, according to Assumption A6 and Lemma~\ref{lemma:consistent}, we have that $(\hat{\theta}-\td)^\top E^{\Qd}\nabla_{\theta}g_{\td}(X) \leq 0$.
Therefore, \eqref{eq:minus-mean} implies
\begin{eqnarray}
(\hat{\theta}-\td)^\top\sum_{i=1}^n\nabla_{\theta} \log g_{\td}(X_i)
\leq 
  (\hat{\theta}-\td)^\top \sum_{i=1}^n[\nabla_{\theta} \log g_{\td}(X_i) - E^{\Qd}\nabla_{\theta}g_{\td}(X_i)].\label{eq:first-order-term}
\end{eqnarray}
Using Chebyshev's inequality and the fact $E(|\nabla_{\theta}\log g_{\td}(X)|^2 )<\infty$, we have
$$
n^{-\frac{3}{4}}\sum_{i=1}^n[\nabla_{\theta} \log g_{\td}(X_i) - E^{\Qd}\nabla_{\theta}g_{\td}(X_i)] \to 0 \mbox{ in probability }\Qd.
$$
According to Slutsky's theorem and $\sqrt{n}(\hat{\theta}-\td)=O_{\Qd}(1)$, we have
$$
\sqrt{n}(\hat{\theta}-\td)^\top n^{-\frac{3}{4}}\sum_{i=1}^n\nabla_{\theta} \log g_{\td}(X_i) \to 0\mbox{ in probability } \Qd.
$$
Consequently,
$$
\lim_{n\to\infty} \Qd\Big((\hat{\theta}-\td)^\top\sum_{i=1}^n[\nabla_{\theta} \log g_{\td}(X_i) - E^{\Qd}\nabla_{\theta}g_{\td}(X_i)] >\frac{\sqrt{n}}{4} \Big)=0.
$$
According to \eqref{eq:first-order-term} and the above display, we have
$$
\lim_{n\to\infty}\Qd\Big((\hat{\theta}-\td)^\top\sum_{i=1}^n\nabla_{\theta} \log g_{\td}(X_i)>\frac{\sqrt{n}}{4} \Big)=0.
$$
Thus, $\Qd(E_2^c)\to 0$ as $n\to \infty$.
We provide an upper bound of $\Qd(E_3^c)$ using a similar method.
With the aid of Chebyshev's inequality, we have
$$n^{-\frac{5}{4}}\sum_{i=1}^n\sup_{\theta} |\nabla^2_{\theta} \log g_{\theta}(X_i)| {\to} 0 \mbox{ in probability } \Qd.$$
According to Slutsky's theorem and  $\sqrt{n}(\hat{\theta}-\td)=O_{\Qd}(1)$, we have
$$
n|\hat{\theta}-\td|^2\times n^{-\frac{5}{4}} \sum_{i=1}^n\sup_{\theta} |\nabla^2_{\theta} \log g_{\theta}(X_i)|\overset{d}{\to} 0.
$$
Consequently,
$$
\lim_{n\to\infty}\Qd\Big(|\hat{\theta}-\td|^2\sup_{\theta} \sum_{i=1}^n |\nabla^2_{\theta} \log g_{\theta}(X_i)|>\frac{\sqrt{n}}{4}\Big)\leq 
\lim_{n\to\infty}\Qd\Big(|\hat{\theta}-\td|^2 \sum_{i=1}^n\sup_{\theta} |\nabla^2_{\theta} \log g_{\theta}(X_i)|>\frac{\sqrt{n}}{4}\Big)=0.
$$
Therefore, $\Qd(E_3)\to 0$ as $n\to \infty$.
We combine the results for $\Qd(E_2^c),\Qd(E_3^c)$, $\Qd(E_4)$, and \eqref{eq:I-lower2},
$$
I\geq \frac{\varepsilon_0}{2} e^{-|\ld|\sqrt{n}} \mbox{ for } n \mbox{ sufficiently large.}
$$
Combining the above display with \eqref{eq: change of measure for lower bound}, we arrive at the lower bound
$$
P_{g_{\theta_0}}(LR_n>e^{nb})\geq e^{-n(1+o(1))\rd}.
$$
We complete the proof by combining the lower bound and upper bound for the probability
$P_{g_{\theta_0}}(LR_n>1)$.

\bigskip
\section{Proof of Theorem~\ref{thm:error-glm}}
The proof is similar to that of Theorem \ref{thm:error-under-theta0}.
Throughout the proof, we will use $\kappa$ as a generic notation to denote large and not-so-important
constants whose value may vary from place to place. Similarly, we use $\varepsilon$ as a generic notation for
small positive constants.
The proof of the theorem consists of establishing upper and lower bounds for the probability
$$
P_{\beta^0}(LR_n\geq1) = P_{\beta^0}\Big(\sup_{\gamma}\inf_{\beta} \sum_{i=1}^n [\log h_i(Y_i, \gamma)-\log g_i(Y_i,\beta)]\geq0\Big).
$$

\paragraph{Upper bound}
Similar to \eqref{eq:upper-bound-first}, we have
$$
P_{\beta^0}(LR_n\geq 1)\leq P_{\beta^0}\Big(\sup_{\gamma} \sum_{i=1}^n [\log h_i(Y_i, \gamma)-\log g_i(Y_i,\beta_n^{\dagger})]\geq0\Big)
$$
According to the definition of $h_i(Y_i,\gamma)$ and $g_i(Y_i,\beta)$, we have 
$$
\sum_{i=1}^n [\log h_i(Y_i, \gamma)-\log g_i(Y_i,\beta)]
= \sum_{i=1}^n [\gamma^T Z^{(i)}Y_i - b(\gamma^T Z^{(i)})] - \sum_{i=1}^n [\beta_n^{\dagger T} X^{(i)}Y_i - b(\beta_n^{\dagger T} X^{(i)})]. 
$$
Consequently, we have
\begin{equation}\label{eq:prob-an}
P_{\beta^0}(LR_n\geq1) \leq  P_{\beta^0}\Big(
(\frac{1}{n}\sum_{i=1}^n Z^{(i)}Y_i,\frac{1}{n}\sum_{i=1}^n X^{(i)}Y_i )\in A_{n}
\Big),
\end{equation}
where
$$
A_{n} = \Big\{
(s_1,s_2): s_1\in R^{p}, s_2\in R^q~\mbox{and}~
\sup_{\gamma} [\gamma^T s_2 -\frac{1}{n}\sum_{i=1}^n b(\gamma^T Z^{(i)}) ]
\geq 
[\beta_n^{\dagger T} s_1 -\frac{1}{n}\sum_{i=1}^n b(\beta_n^{\dagger T} X^{(i)}) ]	
\Big\}.
$$
We consider the change of measure
\begin{equation}\label{eq:relative-drivative}
\frac{dQ^{\dagger}}{dP} 
= 
\exp\Big\{
\lambda_n^{\dagger } \sum_{i=1}^{n}(\gamma_n^{\dagger T}Z^{(i)}Y_i- \beta_n^{\dagger T}X^{(i)}Y_i) - \sum_{i=1}^n[b(({\beta^{0}})^T X^{(i)}+\lambda_n^{\dagger}\{\gamma_n^{\dagger T} Z^{(i)}-\beta_n^{\dagger T}X^{(i)}\}) -b({(\beta^{0})}^{T}X^{(i)})] 
\Big\}.
\end{equation}
According to \eqref{eq:prob-an}, we have
$$
P_{\beta^0}(LR_n\geq1)\leq E^{Q^\dagger}\Big[
\frac{dP}{dQ^{\dagger}}; (\frac{1}{n}\sum_{i=1}^n Z^{(i)}Y_i,\frac{1}{n}\sum_{i=1}^n X^{(i)}Y_i )\in A_{n}.
\Big]
$$
The above display and \eqref{eq:relative-drivative} together gives
\begin{multline}\label{eq:upper-change}
P_{\beta^0}(LR_n\geq1) \leq \exp\Big\{\sum_{i=1}^n\Big[b\Big({(\beta^{0})}^T X^{(i)}+\lambda_n^{\dagger}\{\gamma_n^{\dagger T} Z^{(i)}-\beta_n^{\dagger T}X^{(i)}\}\Big) -b\Big({(\beta^{0})}^{T}X^{(i)}\Big)\Big]\Big\}\\
\times 
E^{Q^{\dagger}}\Big[e^{-\lambda_n^{\dagger } \sum_{i=1}^{n}(\gamma_n^{\dagger T}Z^{(i)}Y_i- \beta_n^{\dagger T}X^{(i)}Y_i)}; (\frac{1}{n}\sum_{i=1}^n Z^{(i)}Y_i,\frac{1}{n}\sum_{i=1}^n X^{(i)}Y_i )\in A_n\Big].
\end{multline}
The next lemma shows a property of $\beta_n^{\dagger}$ and $A_n$.
	\begin{lemma}\label{lemma:convex-an}
		For all $(s_1,s_2)\in A_n$,
		$$
		[\gamma^{\dagger T} s_2 -\frac{1}{n}\sum_{i=1}^n b(\gamma^{\dagger T} Z^{(i)}) ]
		\geq 
		[\beta_n^{\dagger T} s_1 -\frac{1}{n}\sum_{i=1}^n b(\beta_n^{\dagger T} X^{(i)}) ].	
		$$
	\end{lemma}
According to Lemma~\ref{lemma:convex-an}, the right-hand side of \eqref{eq:upper-change} is further bounded above by
\begin{multline}
P_{\beta^0}(LR_n\geq1)
\\
 \leq \exp\Big\{\sum_{i=1}^n\Big[b\Big({(\beta^{0})}^T X^{(i)}+\lambda_n^{\dagger}\{\gamma_n^{\dagger T} Z^{(i)}-\beta_n^{\dagger T}X^{(i)}\}\Big) -b({(\beta^{0})}^{T}X^{(i)})\Big]
-\lambda_n^{\dagger}\sum_{i=1}^n\Big[b\Big(\gamma_n^{\dagger T}Z^{(i)}\Big)-b\Big(\beta_n^{\dagger T}X^{(i)}\Big)\Big]\Big\} \\
\times Q^{\dagger}\Big[(\frac{1}{n}\sum_{i=1}^n Z^{(i)}Y_i,\frac{1}{n}\sum_{i=1}^n X^{(i)}Y_i )\in A_n\Big].
\end{multline}
Because $Q^{\dagger}\Big[(\frac{1}{n}\sum_{i=1}^n Z^{(i)}Y_i,\frac{1}{n}\sum_{i=1}^n X^{(i)}Y_i )\in A_n\Big]\leq 1$, we arrive at
$$
P_{\beta^0}(LR_n\geq1) \leq \exp\Big\{\sum_{i=1}^n[b({(\beta^{0})}^T X^{(i)}+\lambda_n^{\dagger}\{\gamma_n^{\dagger T} Z^{(i)}-\beta_n^{\dagger T}X^{(i)}\}) -b({(\beta^{0})}^{T}X^{(i)})]-\lambda_n^{\dagger}\sum_{i=1}^n[b(\gamma_n^{\dagger T}Z^{(i)})-b(\beta_n^{\dagger T}X^{(i)})]\Big\}.
$$
According to the definition of $\widetilde{\rho}_n^{\dagger}$, the right-hand side of the above inequality equals $e^{-n\widetilde{\rho}_n^{\dagger}}$. Therefore, we arrive at the upper bound 
$$
P_{\beta^0}(LR_n\geq 1) \leq e^{-n \widetilde{\rho}_n^{\dagger}}.
$$
\paragraph{Lower bound}
Notice that the event
$$
\{\sum_{i=1}^n \log h_i(Y_i,\gamma_n^{\dagger}) - \sup_{\beta}\sum_{i=1}^n\log g_i(Y_i,\beta) \geq 0 \}.
$$
implies the event
$$
\{\sup_{\gamma}\sum_{i=1}^n \log h_i(Y_i,\gamma) - \sup_{\beta}\sum_{i=1}^n\log g_i(Y_i,\beta) \geq 0 \}.
$$
Therefore, a lower bound for the probability $P_{\beta^0}(LR_n\geq 1)$ is
$$
P_{\beta^0}\Big(\sum_{i=1}^n \log h_i(Y_i,\gamma_n^{\dagger}) - \sup_{\beta}\sum_{i=1}^n\log g_i(Y_i,\beta) \geq 0 \Big).
$$
According to the definition of $Q^{\dagger}$ in \eqref{eq:relative-drivative}, the above probability equals
\begin{multline}\label{eq:lower-bound}
\exp\Big\{\sum_{i=1}^n[b({(\beta^{0})}^T X^{(i)}+\lambda_n^{\dagger}\{\gamma_n^{\dagger T} Z^{(i)}-\beta_n^{\dagger T}X^{(i)}\}) -b({(\beta^{0})}^{T}X^{(i)})]\Big\}\\
\times E^{Q^{\dagger}}\Big[e^{-\lambda_n^{\dagger } \sum_{i=1}^{n}(\gamma_n^{\dagger T}Z^{(i)}Y_i- \beta_n^{\dagger T}X^{(i)}Y_i)};E\Big],
\end{multline}
where the event
$$
E=\Big\{\sum_{i=1}^{n}\gamma_n^{\dagger T}Z^{(i)}Y_i- \hat{\beta}_n^T X^{(i)}Y_i- b(\gamma_n^{\dagger T}X^{(i)})+ b(\hat{\beta}_n^TX^{(i)})\geq 0\Big\},
$$
and $\hat{\beta}_n$ is the maximum likelihood estimator 
$$
\hat{\beta}_n=\arg\sup_{\beta} \sum_{i=1}^n \beta^T X^{(i)}Y_i -b(\beta X^{(i)}).
$$
Notice that
$$
e^{-n\widetilde{\rho}_n^{\dagger}}= \exp\Big\{\sum_{i=1}^n[b({(\beta^{0})}^T X^{(i)}+\lambda_n^{\dagger}\{\gamma_n^{\dagger T} Z^{(i)}-\beta_n^{\dagger T}X^{(i)}\}) -b({(\beta^{0})}^{T}X^{(i)})]-\lambda_n^{\dagger}[b(\gamma_n^{\dagger T} Z^{(i)})-b(\beta_n^{\dagger T}X^{(i)})] \Big\}.
$$
Therefore, 
\begin{equation}\label{eq:lower-bound-glm}
P_{\beta^0}(LR_n\geq1)\geq 
e^{-n\widetilde{\rho}_n}\times J,
\end{equation}
where we define the quantity
$$
J=
E^{Q^{\dagger}}\Big[e^{-\lambda_n^{\dagger } [\sum_{i=1}^{n}\gamma_n^{\dagger T}Z^{(i)}Y_i- \beta_n^{\dagger T}X^{(i)}Y_i- b(\gamma_n^{\dagger T}X^{(i)})+ b(\beta_n^{\dagger T}X^{(i)})]}; E\Big].$$
We proceed to establishing a lower bound of $J$.
We consider two events
$$
E_1= \Big\{ 
\frac{\sqrt{n}}{2}<\sum_{i=1}^{n}\gamma_n^{\dagger T}Z^{(i)}Y_i- \beta_n^{\dagger T}X^{(i)}Y_i- b(\gamma_n^{\dagger T}X^{(i)})+ b(\beta_n^{\dagger T}X^{(i)}) \leq \sqrt{n}
\Big\}
$$
and
$$
E_2= \Big\{\sum_{i=1}^n[\hat{\beta}_n^T X^{(i)}Y_i -\beta_n^{\dagger T} X^{(i)}Y_i -b(\hat{\beta}_n^TX^{(i)}) +b(\beta_n^{\dagger T}X^{(i)})] \leq \frac{\sqrt{n}}{2} \Big\}.
$$
Because $E_1$ together with $E_2$ implies $E$, we have $E\supset E_1\cap E_2$. Consequently, 
$$
J\geq E^{Q^{\dagger}}\Big[e^{-\lambda_n^{\dagger } [\sum_{i=1}^{n}\gamma_n^{\dagger T}Z^{(i)}Y_i- \beta_n^{\dagger T}X^{(i)}Y_i- b(\gamma_n^{\dagger T}X^{(i)})+ b(\beta_n^{\dagger T}X^{(i)})]}; E_1\cap E_2\Big].
$$
Notice that on the set $E_1$, 
$\sum_{i=1}^{n}\gamma_n^{\dagger T}Z^{(i)}Y_i- \beta_n^{\dagger T}X^{(i)}Y_i- b(\gamma_n^{\dagger T}X^{(i)})+ b(\beta_n^{\dagger T}X^{(i)})\leq \sqrt{n}$. Therefore, 
\begin{equation}\label{eq:J-lower}
J\geq e^{-\lambda_n^{\dagger} \sqrt{n}} Q^{\dagger}(E_1\cap E_2)\geq e^{-\lambda_n^{\dagger} \sqrt{n}}\Big(Q^{\dagger}(E_1)- Q^{\dagger}(E_2^c)\Big).
\end{equation}
We provide an upper bound for $Q^{\dagger}(E_1)$ and a lower bound for $ Q^{\dagger}(E_2^c)$.
\begin{lemma}\label{lemma:clt}
Let
$$
v_n= 
Var^{Q^{\dagger}}\Big(\sum_{i=1}^{n}\gamma_n^{\dagger T}Z^{(i)}Y_i- \beta_n^{\dagger T}X^{(i)}Y_i- b(\gamma_n^{\dagger T}X^{(i)})+ b(\beta_n^{\dagger T}X^{(i)})\Big), 
$$
then 
$v_n=O(n)$ as $n\to\infty$.
Furthermore, we have
$$
\mathcal{L}\Big(v_{n}^{-\frac{1}{2}}\Big[\sum_{i=1}^{n}\gamma_n^{\dagger T}Z^{(i)}Y_i- \beta_n^{\dagger T}X^{(i)}Y_i- b(\gamma_n^{\dagger T}X^{(i)})+ b(\beta_n^{\dagger T}X^{(i)})\Big]\Big)\to N(0,1).
$$
Here, $\mathcal{L}(\cdot)$ denotes the law of random variables and $N(0,1)$ is the distribution of standard normal.	
\end{lemma}
According to Lemma~\ref{lemma:clt}, there exists a constant $\varepsilon>0$ such that
\begin{equation}\label{eq:e1-upper}
Q^{\dagger}(E_1)\geq \varepsilon.
\end{equation}
We proceed to a lower bound for $Q^{\dagger}(E_2)$. Define the function for $\mu\in R^p$
$$
u(\mu, \beta) = (\beta-\beta_n^{\dagger})^T \mu -\sum_{i=1}^n [b(\beta^T X^{(i)}) -b(\beta_n^{\dagger}X^{(i)}) ].
$$
We further define the function
$$
v(\mu)= \sup_{\beta} u(\mu,\beta).
$$
\begin{lemma}\label{lemma:expand-u}
	Let
	\begin{equation*}\label{eq:def-s1d}
	\mu^{\dagger}=\sum_{i=1}^n b'\Big(\lambda_n^{\dagger}(\gamma_n^{\dagger T}Z^{(i)}-\beta_n^{\dagger} X^{(i)})+{(\beta^{0})}^{T} X^{(i)}\Big) X^{(i)},
	\end{equation*}
	then  $v(\mu)$ is twice continuous differentiable around $\mu^{\dagger}$, with $v(\mu^{\dagger})=0 $ and $\nabla v(\mu^{\dagger})=0$. Moreover, we have
	$$
	\nabla^2 v(\mu)= \Big[\sum_{i=1}^n b''\Big(\beta(\mu)^{T} X^{(i)}\Big)X^{(i)}X^{(i)T}\Big]^{-1},
	$$
	where $\beta(\mu) = \arg\sup_{\beta} u(\mu,\beta)$.
\end{lemma}
According to Lemma~\ref{lemma:expand-u} and Taylor expansion of $v(\mu)$ around $\mu^{\dagger}$, we have
\begin{equation}\label{eq:event-bound}
\Big\{v(\mu)\geq \frac{\sqrt{n}}{2} \Big\}\subset \Big\{
\frac{1}{2}\|\mu-\mu^{\dagger}\|^2 \|\nabla^2 v(\mu^{\dagger}) \|_2 \geq \frac{\sqrt{n}}{2}
\Big\},
\end{equation}
where $\|\cdot \|_2$ is denotes the spectral norm of matrices.
According to Lemma~\ref{lemma:expand-u} and Assumptions A10 and A11, $\|\nabla^2 v(\mu^{\dagger}) \|_2 =O(n) $. Therefore, \eqref{eq:event-bound} implies
$$
\Big\{v(\mu)\geq \frac{\sqrt{n}}{2} \Big\}\subset \Big\{ \|\mu-\mu^{\dagger}\|\geq \varepsilon n^{\frac{3}{4}} \Big\}.
$$
Notice that the event $E_2^c = \{v(\sum_{i=1}^n X^{(i)}Y_i)\geq \frac{\sqrt{n}}{2} \}$, we have
$$
Q^{\dagger}(E_2^c)\leq Q^{\dagger}\Big(\|\sum_{i=1}^n X^{(i)}Y_i-\mu^{\dagger}\|\geq \varepsilon n^{\frac{3}{4}} \Big).
$$
With the aid of Chebyshev's inequality, the above display implies
$$
Q^{\dagger}(E_2^c)\leq (\varepsilon^{-2} n^{-\frac{3}{2}}) E^{Q^{\dagger}} \|\sum_{i=1}^n X^{(i)}Y_i-\mu^{\dagger}\|^2
$$
Because $E^{Q^{\dagger}} \|\sum_{i=1}^n X^{(i)}Y_i-\mu^{\dagger}\|^2= O(n)$, we have
$Q^{\dagger}(E_2^c)$ tend to zero as $n$ goes to infinity. Combining this result with \eqref{eq:J-lower} and \eqref{eq:e1-upper}, we arrive at a lower bound for $J$
$$
J\geq \frac{\varepsilon}{2} e^{-\lambda_n^{\dagger}\sqrt{n}}.
$$
The above inequality together with \eqref{eq:lower-bound-glm} gives a lower bound
\begin{equation}\label{eq:lower}
P(LR_n\geq1)\geq \frac{\varepsilon}{2} e^{-n\widetilde{\rho}_n^{\dagger}-\lambda_n^{\dagger}\sqrt{n}}.
\end{equation}
According to Assumption A9, $\widetilde{\rho}_n^{\dagger}\geq \inf_{\gamma}\sup_{\lambda}\widetilde{\rho}_n(\beta^0,\gamma,\lambda)\geq \delta_1$, so $\lambda_n^{\dagger}\sqrt{n}=o(1)n\widetilde{\rho}_n^{\dagger}$. Therefore, \eqref{eq:lower} implies $P_{\beta^0}(LR_n\geq 1)\geq e^{-n\widetilde{\rho}_n^{\dagger}(1+o(1))} .$
We complete the proof by combining the lower and upper bound for $P_{\beta^0}(LR_n\geq 1)$

\section{Proof of Lemma~\ref{lemma:consistent}}\label{sec:proof-for-lemma}
\begin{proof}[Proof of Lemma~\ref{lemma:consistent}]
According to condition A6, it is sufficient to show that for all $y\in T_{\gd}\Gamma$,
\begin{equation}\label{eq:check-gamma-condition}
E^{\Qd}y^\top\nabla_{\gamma} h_{\gd}(X)\leq 0,
\end{equation}
and for all $y\in T_{\td}\Theta $,
\begin{equation}\label{eq:check-theta-condtion}
E^{\Qd} y ^\top\nabla_{\theta} g_{\td}(X)\leq 0.
\end{equation}
We first prove \eqref{eq:check-gamma-condition}. 
We discuss two cases: $\gd\in int(\Gamma)$ and $\gd\in \partial \Gamma$, where $int(\Gamma)$ denotes the interior of $\Gamma$.
\paragraph{Case 1: $\gd\in int(\Gamma)$} Because $\ld=\arg\inf_{\lambda}M_{g_{\theta_0}}(\td,\gd,\lambda)$, we have $\frac{\partial }{\partial \lambda}M_{g_{\theta_0}}(\td,\gd,\ld)=0$. According to the definition of $\gd$, $(\gd,\ld)$ is a  solution of the constrained optimization problem,
\begin{equation}\label{eq:opt-gamma}
\max_{\gamma,\lambda} M_{g_{\theta_0}}(\td,\gamma,\lambda) \mbox{ such that } \frac{\partial }{\partial \lambda}M_{g_{\theta_0}}(\td,\gamma,\lambda)=0,
\end{equation}
and thus it satisfies the Karush-Kuhn-Tucker conditions. That is, there exists a constant $\mu$ such that
$$
\left\{
\begin{array}{lcl}
\nabla_{\gamma}M_{g_{\theta_0}}(\td,\gd,\ld)&=&\mu \nabla_{\gamma}\frac{\partial }{\partial \lambda}M_{g_{\theta_0}}(\td,\gd,\ld)\\
\frac{\partial }{\partial \lambda}M_{g_{\theta_0}}(\td,\gd,\ld)&=& \mu \frac{\partial^2 }{\partial^2 \lambda}M_{g_{\theta_0}}(\td,\gd,\ld)\\
\frac{\partial }{\partial \lambda}M_{g_{\theta_0}}(\td,\gd,\ld)&=&0
\end{array}.
\right.
$$
The second and third equations in the above display together imply that $\mu=0$. We plug $\mu=0$ to the first equation and obtain that
\begin{equation}\label{eq:case1-first-order}
\nabla_{\gamma}M_{g_{\theta_0}}(\td,\gd,\ld)=0.
\end{equation}
According to the definition of $M_{g_{\theta_0}}(\theta,\gamma,\lambda)$, we  have
\begin{equation}\label{eq:gradient-M-gamma}
\nabla_{\gamma}M_{g_{\theta_0}}(\theta,\gamma,\lambda)= \lambda E_{g_{\theta_0}} \exp\{\lambda(\log h_{\gamma}(X)-\log g_{\theta}(X)-b)\}\nabla_{\gamma}\log h_{\gamma}(X)/M_{g_{\theta_0}}^{\dagger}.
\end{equation}
We plug this in \eqref{eq:case1-first-order}, and obtain
$$
E^{\Qd} \nabla_{\gamma}\log h_{\gd}(X)=0.
$$
Consequently, for all $y\in R^{d_{h}}$, \eqref{eq:check-gamma-condition} holds.
\paragraph{Case 2: $\gd\in \partial \Gamma$} Because $\partial \Gamma$ is continuously differentiable, with possibly relabeling the coordinate of $\gamma$, there exists a continuously differentiable function $v:R^{d_{h}-1}\to R$ and $r>0$ such that
\begin{equation}\label{eq:D-v}
B(\gd,r)\cap \Gamma=\{\gamma\in B(\gd,r):\gamma_{d_{h}}\geq v(\gamma_1,...,\gamma_{d_{h}-1})\},
\end{equation}
where $B(\gd,r)=\{\gamma: |\gamma-\gd|\leq r \} $ is a closed ball centered around $\gd$. Similar to Case 1, we consider the constrained optimization problem \eqref{eq:opt-gamma} with the additional constraint
$$
\gamma_{d_{h}}\geq v(\gamma_1,...,\gamma_{d_{h}-1}).
$$
The definition of $\gd$ implies that $(\gd,\ld)$ is a local maximum to this optimization problem. Again, it  satisfies the Karush-Kuhn-Tucker conditions for optimization problem with inequality constraint. That is, there exists constant $\mu_1$ and $\mu_2$ such that $\mu_1\geq 0$ and
$$
\left\{
\begin{array}{lcl}
\frac{\partial}{\partial \gamma_i}M_{g_{\theta_0}}(\td,\gd,\ld)&=&\mu_1 \frac{\partial}{\partial \gamma_i} v(\gd_1,...,\gd_{d_{h}-1}) + 
\mu_2 \nabla_{\gamma}\frac{\partial }{\partial \lambda}M_{g_{\theta_0}}(\td,\gd,\ld) \mbox{ for } i=1,...,d_{h}-1\\
\frac{\partial}{\partial \gamma_d}M_{g_{\theta_0}}(\td,\gd,\ld)&=&-\mu_1  + 
\mu_2 \nabla_{\gamma}\frac{\partial }{\partial \lambda}M_{g_{\theta_0}}(\td,\gd,\ld) \\
\frac{\partial }{\partial \lambda}M_{g_{\theta_0}}(\td,\gd,\ld)&=& \mu_2 \frac{\partial^2 }{\partial^2 \lambda}M_{g_{\theta_0}}(\td,\gd,\ld)\\
\frac{\partial }{\partial \lambda}M_{g_{\theta_0}}(\td,\gd,\ld)&=&0
\end{array}.
\right.
$$
Similar to the Case 1, the third and the fourth equalities together imply that $\mu_2=0$. We plug this in the first and the second equalities and obtain that
\begin{equation}\label{eq:case2-first-order}
\nabla_{\gamma}M_{g_{\theta_0}}(\td,\gd,\ld)= \mu_1 (\nabla v(\gd_{1},...,\gd_{d_{h}-1})^{T},-1)^{T}.
\end{equation}
We now prove that $\gd$ satisfies \eqref{eq:check-gamma-condition}. Notice that $\partial \Gamma$ is continuously differentiable, therefore the tangent cone is
$$
T_{\gd}\Gamma=\{y\in R^{d_{h}}: y\cdot (\nabla v(\gd_1,...,\gd_{d_{h}-1})^T,-1)^T \leq 0 \}.
$$
Consequently, for all $y\in T_{\gd}\Gamma$, \eqref{eq:case2-first-order} implies
\begin{equation}\label{eq:gradiant-gamma}
\nabla_{\gamma}M_{g_{\theta_0}}(\td,\gd,\ld)\cdot y = \mu_1 y\cdot (\nabla v(\gd_1,...,\gd_{d_h-1})^T,-1)^T\leq 0.
\end{equation}
Notice that $$\frac{\partial}{\partial \lambda} M_{g_{\theta_0}}(\td,\gd,0)= E_{g_{\theta_0}}[\log h_{\gd}(X)-\log g_{\td}(X)-b]<0,$$ and $$\frac{\partial^2}{\partial^2 \lambda}M_{g_{\theta_0}}(\td,\gd,\lambda)= E_{g_{\theta_0}}\Big\{e^{\lambda[\log h_{\gamma}(X)-\log g_{\theta}(X)-b]}[\log h_{\gd}(X)-\log g_{\td}(X)-b]^2\Big\}>0.$$
Thus $\ld>0$.
We prove \eqref{eq:check-gamma-condition} by plugging \eqref{eq:gradient-M-gamma} in \eqref{eq:gradiant-gamma} and notice that $\ld>0$.
Now we proceed to the proof of \eqref{eq:check-theta-condtion}. Again, we consider two cases: $\gd\in int(\Gamma)$ and $\gd\in \partial \Gamma$.
\paragraph{Case 1: $\gd\in int(\Gamma)$.} According to the definition of $(\td,\gd,\ld)$ and \eqref{eq:case1-first-order}, $(\td,\gd,\ld)$ is a local minimum of the optimization problem
$$
\inf_{\theta,\gamma,\lambda} M_{g_{\theta_0}}(\theta,\gamma,\lambda) \mbox{ such that } \frac{\partial }{\partial \lambda}M_{g_{\theta_0}}(\theta,\gamma,\lambda)=0, \mbox{ and }\nabla_{\gamma}M_{g_{\theta_0}}(\theta,\gamma,\lambda)=0.
$$
We prove \eqref{eq:check-theta-condtion}  using a similar proof as that for \eqref{eq:check-gamma-condition} and treating $\td$ and $(\gd,\ld)$ as $\gd$ and $\ld$ respectively. The details are omitted.
\paragraph{Case 2: $\gd\in \partial\Gamma$.}
We will first transform the Case 2 to Case 1. Recall the definition of $v$ and $r$ in  \eqref{eq:D-v}, for $\gamma\in B(\gd,r)\cap \partial \Gamma$, we have
 $$\gamma_d=v(\gamma_1,...,\gamma_{d_h-1}).$$
Let $\Phi :R^{d_h}\to R^{d_{h}-1}
$ be a function such that 
$
\Phi(\gamma)=(\gamma_1,...,\gamma_{d_h-1})^T.
$
Let $\xi=\Phi(\gamma)$, and $\xd=\Phi(\gd)$, then for $\gamma\in B(\gd,r)\cap \partial \Gamma $, $\gamma=(\xi^T,v(\xi))^T$.
 We abuse the notation a little and write 
$$
\widetilde{M}_{g_{\theta_0}}(\theta,\xi,\lambda )= M_{g_{\theta_0}}(\theta,\gamma, \lambda),
$$
where $\gamma= (\xi^T,v(\xi))^T$. We further let $\Xi=\Phi(B(\gd,r)\cap \Gamma) $. We compute the partial derivatives of $\widetilde{M}_{g_{\theta_0}}(\theta,\xi,\lambda)$ at $(\td,\xd,\ld)$,
\begin{equation}\label{eq:nabla-xi}
\frac{\partial}{\partial \lambda}\widetilde{M}_{g_{\theta_0}}(\td,\xd,\ld)= \frac{\partial}{\partial \lambda}{M}_{g_{\theta_0}}(\td,\gd,\ld) =0, \mbox{ and }\nabla_{\xi}\widetilde{M}_{g_{\theta_0}}(\td,\xd,\ld)= \frac{d\gamma}{d\xi}(\xd) ^T \nabla_{\gamma} {M}_{g_{\theta_0}}(\td,\gd,\ld),
\end{equation}
where $\frac{d\gamma}{d\xi}$ is a $d_{h}\times (d_{h}-1)$ Jacobian matrix 
$$
\frac{d\gamma}{d\xi}=
\begin{bmatrix}
I_{d_h-1}\\
\nabla v(\xd)^T
\end{bmatrix},
$$
and $I_{d_h-1}$ is the $(d_h-1)\times (d_h-1)$ identity matrix.
We plug \eqref{eq:case2-first-order} and the above expression in \eqref{eq:nabla-xi}, and obtain
$$
\nabla_{\xi}\widetilde{M}_{g_{\theta_0}}(\td,\xd,\ld)= \mu_1 (\nabla v(\xd) -\nabla v(\xd))^T=0.
$$
Therefore, $(\td,\xd,\ld)$ is a local minimum under the constrained optimization problem
$$
\inf_{\theta,\xi,\lambda}\widetilde{M}_{g_{\theta_0}}(\theta,\xi,\lambda) \mbox{ such that } \nabla_{\xi}\widetilde{M}_{g_{\theta_0}}(\theta,\xi,\lambda)=0 \mbox{ and } \frac{\partial}{\partial \lambda}\widetilde{M}_{g_{\theta_0}}(\theta,\xi,\lambda)=0.
$$ 
We complete the proof by replacing $\gamma$ and $\Gamma$ by $\xi$ and $\Xi$ respectively in the proof for Case 1.
\end{proof}
\section{Proof of Lemma~\ref{lemma:convex-an}}
Define the function
$$
w(s_1,s_2)= \sup_{\gamma} [\gamma^T s_2 -\frac{1}{n}\sum_{i=1}^n b(\gamma^T Z^{(i)}) ]
-
[\beta_n^{\dagger T} s_1 -\frac{1}{n}\sum_{i=1}^n b(\beta_n^{\dagger T} X^{(i)}) ].
$$
Then $A_n=\{(s_1,s_2): w(s_1,s_2)\geq 0\}$. 
Let $$s_1^{\dagger}= \frac{1}{n}\sum_{i=1}^n b'\Big(\lambda_n^{\dagger}(\gamma_n^{\dagger T}Z^{(i)}-\beta_n^{\dagger T}X^{(i)})+{(\beta^{0})}^{T}X^{(i)} \Big) X^{(i)}\mbox{ and  }s_2^{\dagger}= \frac{1}{n}\sum_{i=1}^n b'\Big(\lambda_n^{\dagger}(\gamma_n^{\dagger T}Z^{(i)}-\beta_n^{\dagger T}X^{(i)})+{(\beta^{0})}^{T}X^{(i)} \Big) Z^{(i)}.
$$
With similar proof as that for \eqref{eq:check-gamma-condition}, we have that $\gamma_n^{\dagger}$ satisfies first order conditions
\begin{multline}\label{eq:first-beta} \nabla_{\gamma}\widetilde{\rho}_n(\beta_n^{\dagger},\gamma_n^{\dagger},\lambda_n^{\dagger})=\lambda_n^{\dagger}\frac{1}{n}\sum_{i=1}^n\Big[ 
	b'(\gamma_n^{\dagger T}Z^{(i)})Z^{(i)}- b'\Big(\lambda_n^{\dagger}(\gamma_n^{\dagger T}Z^{(i)}-\beta_n^{\dagger T}X^{(i)})+{(\beta^{0})}^{T}X^{(i)} \Big) Z^{(i)}
	\Big]=0_q.
\end{multline}
\eqref{eq:first-beta} is also the first order condition for the optimization problem
$$
\sup_{\gamma} [\gamma^T s_2^{\dagger} -\frac{1}{n}\sum_{i=1}^n b(\gamma^T Z^{(i)}) ]
-
[\beta_n^{\dagger T} s_1^{\dagger} -\frac{1}{n}\sum_{i=1}^n b(\beta_n^{\dagger T} X^{(i)}) ].
$$
Notice that this optimization is concave in $\gamma$. Therefore, $\gamma_n^{\dagger}$ is a solution of the above optimization problem, and
\begin{multline}\label{eq:s-gamma}
w(s_1^{\dagger},s_2^{\dagger}) = 
\sup_{\gamma} [\gamma^T s_2^{\dagger} -\frac{1}{n}\sum_{i=1}^n b(\gamma^T Z^{(i)}) ]
-
[\beta_n^{\dagger T} s_1^{\dagger} -\frac{1}{n}\sum_{i=1}^n b(\beta_n^{\dagger T} X^{(i)}) ]
\\
= \gamma_n^{\dagger T} s_2^{\dagger} -\frac{1}{n}\sum_{i=1}^n b(\gamma_n^{\dagger T} Z^{(i)}) 
-
[\beta_n^{\dagger T} s_1^{\dagger} -\frac{1}{n}\sum_{i=1}^n b(\beta_n^{\dagger T} X^{(i)}) ]
.
\end{multline}
Also notice that $\lambda_n^{\dagger}=\arg\sup_{\lambda}\widetilde{\rho}_n(\beta_n^{\dagger},\gamma_n^{\dagger},\lambda)$. Therefore, it satisfies the first order condition
\begin{multline}\label{eq:first-lam}
0= \frac{\partial}{\partial \lambda}\widetilde{\rho}_n(\beta_n^{\dagger},\gamma_n^{\dagger},\lambda_n^{\dagger})=\frac{1}{n}\sum_{i=1}^n b(\gamma_n^{\dagger} Z^{(i)})-b(\beta_n^{\dagger}X^{(i)}) - b'\Big(\lambda_n^{\dagger}(\gamma_n^{\dagger T}Z^{(i)}-\beta_n^{\dagger T}X^{(i)})+{(\beta^{0})}^{T}X^{(i)} \Big) [\gamma_n^{\dagger T}Z^{(i)}-\beta_n^{\dagger T}X^{(i)}].
\end{multline}
\eqref{eq:s-gamma} and \eqref{eq:first-lam} together gives
$$
w(s_1^{\dagger},s_2^{\dagger})=0.
$$
Therefore, $(s_1^{\dagger},s_2^{\dagger})$ is a boundary point of $A_n$.
Furthermore, we have
\begin{equation*}\label{eq:normal-vec}
\nabla_{s_1} w(s_1^{\dagger},s_2^{\dagger}) =0_p\mbox{ and } \nabla_{s_2} w(s_1^{\dagger},s_2^{\dagger}) =\beta_n^{\dagger}.
\end{equation*}
Consequently, the normal vector of $A_n$ at $(s_1^{\dagger},s_2^{\dagger})$ is $-(\nabla_{s_1}w(s_1^{\dagger},s_2^{\dagger}),\nabla_{s_2}w(s_1^{\dagger},s_2^{\dagger}))=- (0_p,\gamma_n^{\dagger})$.
Because $A_n$ is a convex set, for all $(s_1,s_2)\in A_n$, we have
$$-(s_1-s_1^{\dagger},s_2-s_2^{\dagger})\cdot (0_p,\gamma_n^{\dagger})\leq 0.
$$
We complete the proof by combining the above display and that $w(s_1^{\dagger},s_2^{\dagger})=0$.
\section{Proof of Lemma~\ref{lemma:clt}}
Because $Y_i$s are independent, we
have
\begin{equation}\label{eq:vn}
v_n=\sum_{i=1}^n
(\gamma_n^{\dagger T}Z^{(i)}
-\beta_n^{\dagger T}X^{(i)}
)^2 Var^{Q^{\dagger}}(Y_i)= \sum_{i=1}^n
(\gamma_n^{\dagger T}Z^{(i)}
-\beta_n^{\dagger T}X^{(i)}
)^2 b''\Big(\lambda_n^{\dagger}(\gamma_n^{\dagger T}Z^{(i)}-\beta_n^{\dagger T}X^{(i)})+{(\beta^{0})}^{T}X^{(i)} \Big).
\end{equation}
According to Assumption A10 and A11, we have $v_n=O(n)$. We define a triangular array for $n,i\geq 1$
$$
U_{n,i}= v_n^{-\frac{1}{2}}\Big[\gamma_n^{\dagger T}Z^{(i)}Y_i- \beta_n^{\dagger T}X^{(i)}Y_i- b(\gamma_n^{\dagger T}X^{(i)})+ b(\beta_n^{\dagger T}X^{(i)})\Big].
$$
It is sufficient to show that $U_{n,i}$ satisfies conditions for the Lyapuvov central limit theorem \cite[page~362]{billingsley1995probability} for triangular arrays.
That is,
\begin{equation}\label{eq:Lyapounov-condition}
\lim_{n\to\infty}\sum_{i=1}^nE^{Q^{\dagger}} |U_{n,i}|^3=0.
\end{equation}
According to Assumption A11, $b(\cdot)$ is four times continuously differentiable. This guarantees that
$$
\sum_{i=1}^n E^{Q^{\dagger}} \Big[\gamma_n^{\dagger T}Z^{(i)}Y_i- \beta_n^{\dagger T}X^{(i)}Y_i- b(\gamma_n^{\dagger T}X^{(i)})+ b(\beta_n^{\dagger T}X^{(i)})\Big]^3 = O(n).
$$ 
Now we show that $v_n^{-1}=O(n^{-1})$. According to Assumptions A10 and A11 and \eqref{eq:def-widetilde-rho}, we have
$$
|\widetilde{\rho}_n^{\dagger}|
\leq \kappa \frac{1}{n}\sum_{i=1}^n |\gamma_n^{\dagger T}Z^{(i)}-\beta_n^{\dagger T}X^{(i)}| \leq \kappa \Big(
\frac{1}{n}\sum_{i=1}^n(\gamma_n^{\dagger T}Z^{(i)}-\beta_n^{\dagger T}X^{(i)})^2	\Big)^{	\frac{1}{2}}.
$$
On the other hand, Assumption A9 implies that
$$
\widetilde{\rho}_n^{\dagger}\geq \inf_{\gamma}\sup_{\lambda}\widetilde{\rho}_n(\beta^0,\gamma,\lambda)\geq \delta_1.
$$
Therefore, $$\frac{1}{n}\sum_{i=1}^n(\gamma_n^{\dagger T}Z^{(i)}-\beta_n^{\dagger T}X^{(i)})^2\geq \delta_1^2\kappa^{-2}.$$ According to Assumption A11 and \eqref{eq:vn}, $v_n\geq \varepsilon \sum_{i=1}^n(\gamma_n^{\dagger T}Z^{(i)}-\beta_n^{\dagger T}X^{(i)})^2 $. Together with the above display, we have
$v_n^{-1}=O(n^{-1})$.
Therefore,
$$
\sum_{i=1}^nE^{Q^{\dagger}} |U_{n,i}|^3= \lim_{n\to\infty} v_n^{-\frac{3}{2}}
\sum_{i=1}^nE^{Q^{\dagger}} \Big[\gamma_n^{\dagger T}Z^{(i)}Y_i- \beta_n^{\dagger T}X^{(i)}Y_i- b(\gamma_n^{\dagger T}X^{(i)})+ b(\beta_n^{\dagger T}X^{(i)})\Big]^3 = O(n^{-\frac{1}{2}}),
$$
and \eqref{eq:Lyapounov-condition} is proved.
\section{Proof of Lemma~\ref{lemma:expand-u}}
Let $
\beta(\mu)=\arg\sup_{\beta} u(\mu,\beta),
$
then $\beta(\mu)$ satisfies
\begin{equation}\label{eq:first-beta-mu}
\frac{\partial}{\partial \beta} u(\mu,\beta(\mu))=0.
\end{equation}
We first show that $\beta(\mu^{\dagger})=\beta_n^{\dagger}$.
Similar to \eqref{eq:check-theta-condtion}, we have
$$
\nabla_{\beta} \widetilde{\rho}_n(\beta_n^{\dagger},\gamma_n^{\dagger},\lambda_n^{\dagger})= \lambda_n^{\dagger}\frac{1}{n}\sum_{i=1}^n\Big[ -b'(\beta_n^{\dagger T}X^{(i)})X^{(i)}+b'\Big(\lambda_n^{\dagger}(\gamma_n^{\dagger T}Z^{(i)}-\beta_n^{\dagger T}X^{(i)})+{(\beta^{0})}^{T}X^{(i)}\Big)X^{(i)}\Big]=0.
$$
Therefore, we have
$$
\frac{\partial}{\partial \beta} u(\mu^{\dagger},\beta_n^{\dagger})= \mu^{\dagger} - \sum_{i=1}^n b'(\beta_n^{\dagger T}X^{(i)})X^{(i)}=0.
$$
Notice that  $\sup_{\beta}u(\mu^{\dagger},\beta)$ is a  strictly concave optimization problem. Therefore, $\beta_n^{\dagger}$ is its unique solution $\beta(\mu)$.
Now we compute $\nabla v(\mu)$.
$$
\nabla v(\mu)= \nabla u(\mu,\beta(\mu))= \frac{\partial}{\partial \mu}u(\mu,\beta(\mu)) +\frac{d}{d\mu}\beta(\mu) \frac{\partial}{\partial \beta} u(\mu,\beta(\mu)).
$$
The above display together with \eqref{eq:first-beta-mu} gives
\begin{equation}\label{eq:first-derivative}
\nabla v(\mu)=\frac{\partial}{\partial \mu}u(\mu,\beta(\mu))=\beta(\mu)-\beta_n^{\dagger}.
\end{equation}
Because $\beta(\mu^{\dagger})=\beta_n^{\dagger}$, we have that $v(\mu)$ is continuously differentiable and $v(\mu^{\dagger})=0$ and $\nabla v(\mu^{\dagger})=0$.
We proceed to the second derivatives of $v(\mu)$.
Applying implicit function theorem to \eqref{eq:first-beta-mu}, we have
$$
\nabla \beta(\mu)= - \frac{\partial^2}{(\partial \beta)^2}u(\mu,\beta)^{-1}\frac{\partial^2}{\partial \mu\partial\beta} u(\mu,\beta(\mu)) = -\frac{\partial^2}{(\partial \beta)^2}u(\mu,\beta)^{-1}.
$$
According to \eqref{eq:first-derivative} and the above equation, we complete the proof.
\bibliographystyle{plainnat}
\bibliography{reference,bibstat}

\end{document}